\documentclass{amsart}
\usepackage{amssymb}
\usepackage[all]{xy}

\newcommand{\IN}{\mathbb N}
\newcommand{\IR}{\mathbb R}
\newcommand{\IQ}{\mathbb Q}
\newcommand{\IZ}{\mathbb Z}
\newcommand{\V}{\mathcal V}
\newcommand{\W}{\mathcal W}
\newcommand{\w}{\omega}
\newcommand{\itau}{\breve\tau}
\newcommand{\ctau}{\bar\tau}
\newcommand{\la}{\langle}
\newcommand{\ra}{\rangle}
\newcommand{\Tau}{\mathcal T}
\newcommand{\IK}{\mathsf K}
\newcommand{\Top}{\mathsf{Top}}
\newcommand{\K}{\mathcal K}

\newtheorem{theorem}{Theorem}[section]
\newtheorem{problem}[theorem]{Problem}
\newtheorem{proposition}[theorem]{Proposition}
\newtheorem{corollary}[theorem]{Corollary}
\newtheorem{lemma}[theorem]{Lemma}

\newtheorem{example}[theorem]{Example}
\theoremstyle{definition}
\newtheorem{definition}[theorem]{Definition}

\title{Kuratowski monoids of $n$-topological spaces}
\author{T.~Banakh, O.~Chervak, T.~Martynyuk, M.~Pylypovych, A.~Ravsky, M.~Simkiv}
\address{T.~Banakh: Ivan Franko National University of Lviv (Ukraine) and Jan Kochanowski University in Kielce (Poland)}
\email{t.o.banakh@gmail.com}
\address{O.~Chervak, T.~Martynyuk, M.~Pylypovych, M.~Simkiv: Ivan Franko National University of Lviv  (Ukraine)}
\email{oschervak@gmail.com, tetyanka.martynyuk@gmail.com, pylypovych@gmail.com, simkiv.markiyan@gmail.com}
\address{A.Ravsky: Institute for Applied Problems of Mechanics and Mathematics, Lviv (Ukraine)}
\email{oravsky@mail.ru}
\dedicatory{Dedicated to the 120-th birthday of K.~Kuratowski (1896-1980)}%\footnote{Kazimierz Kuratowski was born on 2 February 1896 in Warsaw.}}

\begin{document}
\begin{abstract} Generalizing the famous 14-set closure-complement Theorem of Kuratowski from 1922, we prove that for a set $X$ endowed with $n$ pairwise comparable topologies $\tau_1\subset\dots\subset\tau_n$, by repeated application of the operations of complement and closure in the topologies $\tau_1,\dots,\tau_n$ to a subset $A\subset X$ we can obtain at most $2K(n)=2\sum_{i,j=0}^n\binom{i+j}{i}\binom{i+j}{j}$ distinct sets.
\end{abstract}
\maketitle

\section{Introduction}

This paper was motivated by the famous Kuratowski 14-set closure-complement Theorem \cite{Kur22}, which says that the repeated application of the operations of closure and complement to a subset $A$ of a topological space $X$ yields at most 14 pairwise distinct sets\footnote{A complete bibliography related to the Kuratowski 14-set closure-complement Theorem is collected on the web-site (http://www.mathtransit.com/cornucopia.php) created by Mark Bowron.}. More precisely, this theorem says that for any topological space $(X,\tau)$ the operators of complement $c:\mathcal P(X)\to\mathcal P(X)$, $c:A\mapsto X\setminus A$ and closure $\bar\tau:\mathcal P(X)\to\mathcal P(X)$, $\bar\tau:A\mapsto \bar A$, generate a submonoid $\la c,\bar\tau\ra$ of cardinality $\le 14$ in the monoid $\mathcal P(X)^{\mathcal P(X)}$ of all self-maps of the power-set $\mathcal P(X)$ of $X$.

In \cite{SW} Shallitt and Willard constructed two commuting closure operators $p,q:\mathcal P(X)\to\mathcal P(X)$ on the power-set $\mathcal P(X)$ of a countable set $X$ such that the submonoid $\la p,q,c\ra\subset\mathcal P(X)^{\mathcal P(X)}$ generated by these closure operators and the operator of complement is infinite.  In Example~\ref{infinite} below we shall define two metrizable topologies $\tau_1$ and $\tau_2$ on a countable set $X$ such that the closure operators $\ctau_1$ and $\ctau_2$ in the topologies $\tau_1$ and $\tau_2$ generate an infinite submonoid $\la\ctau_1,\ctau_2\ra$ in the monoid $\mathcal P(X)^{\mathcal P(X)}$ of all self-maps of $\mathcal P(X)$. Moreover, for some set $A\subset X$ the set $\{f(A):f\in\la\ctau_1,\ctau_2\ra\}$ is infinite. This shows that Kuratowski's 14-set theorem does not generalize to spaces endowed with two or more topologies.

The situation changes dramatically if two topologies $\tau_1$ and $\tau_2$ on a set $X$ are comparable, i.e., one of these topologies is contained in the other. In this case we shall prove that the closure operators $\ctau_1,\ctau_2:\mathcal P(X)\to\mathcal P(X)$ induced by these topologies together with the operator $c$ of complement generate a submonoid $\la \ctau_1,\ctau_2,c\ra\subset\mathcal P(X)$ of cardinality $\le 126$. In fact, we shall consider this problem in a more general context of multitopological spaces and polytopological spaces.

By a {\em multitopological space} we understand a set $X$ endowed with a family $\mathcal T$ of topologies on $X$. A multitopological space $(X,\Tau)$ is called {\em polytopological} if the family of its topologies  $\Tau$ is linearly ordered by the inclusion relation.  A typical example of a polytopological space is the real line endowed with the Euclidean and Sorgenfrey topologies. Another natural example of a polytopological space is any Banach space, carrying the norm and weak topologies.
A dual Banach space is an example of a polytopological space carrying three topologies: the norm topology, the weak topology and the $*$-weak topology. A topological space $(X,\tau)$ can be thought as a polytopological space $(X,\{\tau\})$ endowed with the family $\{\tau\}$ consisting of a single topology $\tau$.

For a topology $\tau$ on a set $X$ by $\itau:\mathcal P(X)\to\mathcal P(X)$ and $\ctau:\mathcal P(X)\to\mathcal P(X)$ we shall denote the operators of taking the interior and closure with respect to the topology $\tau$. These operators assign to each subset $A\subset X$ its interior $\itau(A)$ and closure $\ctau(A)$, respectively. Since $\tau=\{\breve\tau(A):A\subset X\}=\{X\setminus\bar\tau(A):A\subset X\}$ the topology $\tau$ can be recovered from the operators $\breve\tau$ and $\bar \tau$.

For a multitopological space $\mathbf X=(X,\Tau)$ the submonoid
$$\IK(\mathbf X)=\la \itau,\ctau:\tau\in \Tau\ra$$in $\mathcal P(X)^{\mathcal P(X)}$ generated by the interior and closure operators $\itau,\ctau$ for $\tau\in\Tau$, will be called {\em the Kuratowski monoid} of the multitopological space $\mathbf X$. A somewhat larger submonoid
$$\IK_2(\mathbf X)=\la c,\ctau:\tau\in\Tau\ra$$ in $\mathcal P(X)^{\mathcal P(X)}$ generated by the operator of complement $c$ and the closure operators $\ctau$, $\tau\in\Tau$, will be called {\em the  full Kuratowski monoid} of the multitopological space $\mathbf X$.

Taking into account that $\itau=c\circ\ctau\circ c$ and $\ctau=c\circ\itau\circ c$, we see that $\IK(\mathbf X)\subset \IK_2(\mathbf X)$ and moreover, $$\IK_2(\mathbf X)=\IK(\mathbf X)\cup  \big(c\circ\IK(\mathbf X)\big),$$which implies that $|\IK_2(\mathbf X)|\le 2\cdot|\IK(\mathbf X)|$.

The notion of  a multitopological space has one disadvantage: multitopological spaces do not form a category (it is not clear what to understand under a morphism of multitopological spaces). This problem with multitopological spaces can be easily fixed by introducing their parametric version called  $L$-topological spaces where $(L,\le)$ is a partially ordered set.

 Given a subset $X$ we denote  by $\Top(X)$ the family  of all possible topologies on $X$, partially ordered by the inclusion relation. The family $\Top(X)$ is a lattice whose smallest element is the anti-discrete topology $\tau_a$ and the largest element is the discrete topology $\tau_d$ on $X$. Observe that for the discrete topology the operators $\breve\tau_d$ and $\bar\tau_d$ coincide with the identity operator $1_X$ on $\mathcal P(X)$.

Let $(L,\le)$ be a partially ordered set. By definition, an {\em $L$-topology} on a set $X$ is any monotone map $\tau:L\to\Top(X)$. The monotonicity of $\tau$ means that for any elements $i\le j$ in $L$ we get $\tau(i)\subset\tau(j)$. In the sequel for an element $i\in L$ it will be convenient to denote the topology $\tau(i)$ by $\tau_i$. By an {\em $L$-topological space} we shall understand a pair $(X,\tau)$ consisting of a set $X$ and an $L$-topology $\tau:L\to \Top(X)$ on $X$.

By a morphism between two $L$-topological spaces $(X,\tau)$ and $(Y,\sigma)$ we understand a map $f:X\to Y$ which is continuous as a map between topological spaces $(X,\tau_i)$ and $(Y,\sigma_i)$ for every $i\in L$.  $L$-Topological spaces and their morphisms form a category called the {\em category of $L$-topological spaces}. Each $L$-topological space $\mathbf X=(X,\tau)$ can be thought as a multitopological space endowed with the family of topologies $\{\tau_i\}_{i\in L}$. If the set $L$ is linearly ordered, then the multitopological space $(X,\{\tau_i\}_{i\in L})$ is polytopological.

So we can speak about the Kuratowski monoid $\IK(\mathbf X)$ and the full Kuratowski monoid $\IK_2(\mathbf X)$ of an $L$-topological space $\mathbf X$.

We shall be especially interested in (full) Kuratowski monoids of $n$-topological spaces where $n=\{0,\dots,n-1\}$ is a finite non-zero ordinal (or a natural number). Observe that $n$-topological spaces can be thought as sets  endowed with $n$-topologies $\tau_0\subset\tau_1\subset\dots\subset\tau_{n-1}$.

We shall prove that the upper bound for the cardinality of the Kuratowski monoid $\IK(X)$ of an $n$-topological space $X$  is given by the number
$$K(n)=\sum_{i,j=0}^{n}\textstyle\binom{i+j}{i}\cdot\binom{i+j}{j}$$where $\binom{n}{i}=\frac{n!}{i!(n-i)!}$ is the binomial coefficient.

The main result of this paper is the following theorem.

\begin{theorem}\label{main} For any $n$-topological space $\mathbf X=(X,\Tau)$ its Kuratowski monoid $\IK(\mathbf X)$ has cardinality $|\IK(\mathbf X)|\le K(n)$ and its full Kuratowski monoid $\IK_2(\mathbf X)$ has cardinality
$|\IK_2(\mathbf X)|\le  2\cdot K(n).$
\end{theorem}

The upper bounds $|\IK(\mathbf X)|\le K(n)$ and
$|\IK_2(\mathbf X)|\le  2\cdot K(n)$ given in Theorem~\ref{main} are exact as shown in our next theorem.

\begin{theorem}\label{exact} For every $n\in\w$ there is an $n$-topological space $\mathbf X=(X,\Tau)$ such that $|\IK(\mathbf X)|=K(n)$ and $|\IK_2(\mathbf X)|=2\cdot K(n)$.
\end{theorem}

This theorem will be proved in Section~\ref{s:separ} (see Corollary~\ref{separateL}). The asymptotics of the sequence $K(n)$ is described in the following theorem, which will be proved in Section~\ref{s:asymp}.

\begin{theorem}\label{asymp} There exists $\lim_{n\to\infty} K(n)/{\binom{2n}{n}^{2}=\sup_{n\to\infty}K(n)}/\binom{2n}{n}^{2}=\frac{16}{9}$ which implies that\newline
$K(n)=\big(\frac{16}{9}+o(1)\big)\cdot\binom{2n}{n}^2=\frac{16^{n+1}}{9\pi n}\cdot\big(1+o(1)\big)$.
\end{theorem}

 The values of the sequences  $K(n)$ and $2K(n)$ for $n\le 9$, calculated with help of computer are presented in the following table:
\bigskip

\begin{tabular}{|c|r|r|r|r|r|r|r|r|r|r|r|}
\hline
\phantom{\large${}^{|}$}$n$&0&1 &2 & 3       & 4     & 5     & 6      & 7    & 8   & 9\\%&10\\
\hline
\phantom{\large${}^{|}$}$K(n)$&1&7&63&697&8549&111033&1495677&20667463&291020283&4157865643\\%&60094478179\\
\hline
\phantom{\large${}^{|}$}$2K(n)$&2&{\bf 14}&126&1394&17098&222066&2991359&41334926&582040566&8315731286\\%&120188956358\\
\hline
\end{tabular}
\medskip

In particular, the Kuratowski monoid $\IK(\mathbf X)$ of any topological space $\mathbf X=(X,\tau)$ consists of 7 elements (some of which can coincide \cite{GJ}):
$$1,\;\;\itau,\;\;\ctau,\;\;\itau\ctau,\;\ctau\itau,\;\;\itau\ctau\itau,\;\ctau\itau\ctau.$$

For a 2-topological space $\mathbf X=(X,\tau)$ endowed with two topologies $\tau_0\subset\tau_1$ the number of elements of the Kuratowski monoid $\IK(\mathbf X)$ increases to 63 (see Proposition~\ref{quot}):
$$\begin{aligned}
&1,\;\;\itau_0,\;\itau_1,\;\ctau_0,\;\ctau_1,\\
&\itau_0\ctau_0,\;\itau_0\ctau_1,\;\itau_1\ctau_0,\;\itau_1\ctau_1,\;
\ctau_0\itau_0,\;\ctau_0\itau_1,\;\ctau_1\itau_0,\;\ctau_1\itau_1,\;\\
&\itau_0\ctau_0\itau_0,\; \itau_0\ctau_0\itau_1,\; \itau_0\ctau_1\itau_0,\;\itau_0\ctau_1\itau_1,\;
\itau_1\ctau_0\itau_0,\; \itau_1\ctau_0\itau_1,\; \itau_1\ctau_1\itau_0,\;\itau_1\ctau_1\itau_1,\;\\
&\ctau_0\itau_0\ctau_0,\; \ctau_0\itau_0\ctau_1,\; \ctau_0\itau_1\ctau_0,\; \ctau_0\itau_1\ctau_1,\;
\ctau_1\itau_0\ctau_0,\; \ctau_1\itau_0\ctau_1,\; \ctau_1\itau_1\ctau_0,\; \ctau_1\itau_1\ctau_1,\;\\
&\itau_0\ctau_0\itau_0\ctau_1,\;\itau_0\ctau_0\itau_1\ctau_1,\;
\itau_1\ctau_0\itau_0\ctau_0,\;\itau_1\ctau_0\itau_0\ctau_1,\;\itau_1\ctau_0\itau_1\ctau_1,\;
\itau_1\ctau_1\itau_0\ctau_0,\;\itau_1\ctau_1\itau_0\ctau_1,\\
&\ctau_0\itau_0\ctau_0\itau_1,\;\ctau_0\itau_0\ctau_1\itau_1,\;
\ctau_1\itau_0\ctau_0\itau_0,\;\ctau_1\itau_0\ctau_0\itau_1,\;\ctau_1\itau_0\ctau_1\itau_1,\;
\ctau_1\itau_1\ctau_0\itau_0,\;\ctau_1\itau_1\ctau_0\itau_1,\\
&\itau_1\ctau_0\itau_0\ctau_0\itau_1,\; \itau_1\ctau_1\itau_0\ctau_1\itau_1,\;
\itau_0\ctau_0\itau_0\ctau_1\itau_1,\; \itau_1\ctau_1\itau_0\ctau_0\itau_0,\; \itau_1\ctau_0\itau_0\ctau_1\itau_1,\;\itau_1\ctau_1\itau_0\ctau_0\itau_1,\; \\
&\ctau_1\itau_0\ctau_0\itau_0\ctau_1,\; \ctau_1\itau_1\ctau_0\itau_1\ctau_1,\;
\ctau_0\itau_0\ctau_0\itau_1\ctau_1,\; \ctau_1\itau_1\ctau_0\itau_0\ctau_0,\; \ctau_1\itau_0\ctau_0\itau_1\ctau_1,\;\ctau_1\itau_1\ctau_0\itau_0\ctau_1,\;\\
&\itau_1\ctau_1\itau_0\ctau_0\itau_1\ctau_1,\; \itau_1\ctau_0\itau_0\ctau_0\itau_1\ctau_1,\;\itau_1\ctau_1\itau_0\ctau_0\itau_0\ctau_1,\\
&\ctau_1\itau_1\ctau_0\itau_0\ctau_1\itau_1,\; \ctau_1\itau_0\ctau_0\itau_0\ctau_1\itau_1,\;\ctau_1\itau_1\ctau_0\itau_0\ctau_0\itau_1,\;\\
&\itau_1\ctau_1\itau_0\ctau_0\itau_0\ctau_1\itau_1,\;\; \ctau_1\itau_1\ctau_0\itau_0\ctau_0\itau_1\ctau_1.
\end{aligned}
$$

\section{The Kuratowski monoid of a saturated polytopological space}

In this section we introduce a class of $n$-topological spaces $\mathbf X=(X,\Tau)$ whose Kuratowski monoids $\IK(\mathbf X)$ have cardinality strictly smaller than $K(n)$.

A multitopological space $(X,\tau)$ is called {\em saturated\/} if for any topologies $\tau_0,\tau_1\in\Tau$ each non-empty open subset $U\in\tau_0$ has non-empty interior in the topology $\tau_1$.
A typical example of a saturated multitopological space is the real line $\IR$ endowed with the 2-element family $\Tau=\{\tau_0,\tau_1\}$ consisting of the Euclidean and Sorgenfrey topologies $\tau_0\subset\tau_1$.

For a linearly ordered set $L$ an $L$-topological space $(X,\tau)$ is defined to be {\em saturated} if the multitopological space $(X,\{\tau_i\}_{i\in L})$ is saturated.

\begin{theorem}\label{satur} For a saturated polytopological space $\mathbf X=(X,\Tau)$ the Kuratowski monoid $\IK(\mathbf X)$ coincides with the set
$$\{1\}\cup\{\itau,\ctau,\;\itau\ctau,\ctau\itau,\;\itau\ctau\itau,\ctau\itau\ctau:\tau\in\Tau\}$$and hence has cardinality $|\IK(\mathbf X)|\le 1+6\cdot|\Tau|$.
\end{theorem}

\begin{proof} The definition of a saturated polytopological space $\mathbf X=(X,\Tau)$ implies that $\ctau_0\itau_1=\ctau_0\itau_0$ for any topologies $\tau_0,\tau_1\in\Tau$.
Applying to this equality the operator $c$ of taking complement, we get
$$\itau_0\ctau_1=c\ctau_0cc\itau_1c=c\ctau_0\itau_1c=c\ctau_0\itau_0c=c\ctau_0cc\itau_0c=\itau_0\ctau_0.$$
This implies that $$\IK(\mathbf X)=\bigcup_{\tau\in\Tau}\IK(X,\tau)=\bigcup_{\tau\in\Tau}\{1,\itau,\ctau,\itau\ctau,\ctau\itau,\itau\ctau\itau,\ctau\itau\ctau\}$$and hence $|\IK(\mathbf X)|\le 1+\sum_{\tau\in\Tau}(|\IK(X,\tau)|-1)\le 1+6\cdot |\Tau|.$
\end{proof}

\begin{example} Let $\mathbf X=(\IR,\{\tau_0,\tau_1\})$ be the real line $\IR$ endowed with the Euclidean topology $\tau_0$ and the Sorgenfrey topology $\tau_1$. The Kuratowski monoid $\IK(\mathbf X)$ has cardinality $|\IK(\mathbf X)|=13$. Moreover, for some set $A\subset\IR$ the family $\IK(\mathbf X)A=\{f(A):A\in \IK(\mathbf X)\}$ has cardinality $|\IK(\mathbf X)A|=|\IK(\mathbf X)|=13$.
\end{example}

\begin{proof} The upper bound $|\IK(\mathbf X)|\le 13$ follows from Theorem~\ref{satur}. To prove the lower bound $|\IK(\mathbf X)|\ge 13$, consider the subset
$$A=\bigcup_{n=0}^\infty[3^{-2n-1},3^{-2n})\cup(1,2)\cup\{3\}\cup\big([4,5)\cap\IQ\big)$$and observe that
the following 13 subsets of $\IR$ are pairwise distinct, witnessing that $|\IK(\mathbf X)|\ge|\{f(A):f\in\IK(\mathbf X)\}|\ge 13$.

\begin{center}
\begin{tabular}{|r|l|}
\hline
\phantom{${}^{\big|}$}$A$&$\hskip26pt\bigcup_{n=0}^\infty[3^{-2n-1},3^{-2n})\cup(1,2)\cup\{3\}\cup\big([4,5)\cap\IQ\big)$\\
$\itau_0(A)$&$\hskip26pt\bigcup_{n=0}^\infty(3^{-2n-1},3^{-2n})\cup(1,2)$\\
$\ctau_0(A)$&$\{0\}\cup\bigcup_{n=0}^\infty[3^{-2n-1},3^{-2n}]\cup[1,2]\cup\{3\}\cup[4,5]$\\
$\ctau_0\itau_0(A)$&$\{0\}\cup\bigcup_{n=0}^\infty[3^{-2n-1},3^{-2n}]\cup[1,2]$\\
$\itau_0\ctau_0(A)$&$\hskip26pt\bigcup_{n=0}^\infty(3^{-2n-1},3^{-2n})\cup[1,2)\cup(4,5)$\\
$\itau_0\ctau_0\itau_0(A)$&$\hskip26pt\bigcup_{n=0}^\infty(3^{-2n-1},3^{-2n})\cup[1,2)$\\
$\ctau_0\itau_0\ctau_0(A)$&$\{0\}\cup\bigcup_{n=0}^\infty[3^{-2n-1},3^{-2n}]\cup[1,2]\cup[4,5]$\\
$\itau_1(A)$&$\hskip26pt\bigcup_{n=0}^\infty[3^{-2n-1},3^{-2n})\cup(1,2)$\\
$\ctau_1(A)$&$\{0\}\cup\bigcup_{n=0}^\infty[3^{-2n-1},3^{-2n})\cup[1,2)\cup\{3\}\cup[4,5)$\\
$\ctau_1\itau_1(A)$&$\{0\}\cup\bigcup_{n=0}^\infty[3^{-2n-1},3^{-2n})\cup[1,2)$\\
$\itau_1\ctau_1(A)$&$\hskip26pt\bigcup_{n=0}^\infty[3^{-2n-1},3^{-2n})\cup[1,2)\cup[4,5)$\\
$\itau_1\ctau_1\itau_1(A)$&$\hskip26pt\bigcup_{n=0}^\infty[3^{-2n-1},3^{-2n})\cup[1,2)$\\
$\ctau_1\itau_1\ctau_1(A)$&$\{0\}\cup\bigcup_{n=0}^\infty[3^{-2n-1},3^{-2n})\cup[1,2)\cup[4,5)$\phantom{${}_{\big|}$}\\
\hline
\end{tabular}
\end{center}
\end{proof}

Theorem~\ref{satur} gives a partial answer to the following general problem.

\begin{problem} Which properties of a polytopological space $X$ are reflected in the algebraic structure of its Kuratowski monoid $\IK(X)$?
\end{problem}

\section{An example of a multitopological space with infinite Kuratowski monoid}

In this section we shall construct the following example announced in the introduction.

\begin{example}\label{infinite} There is a countable space $X$ endowed with two (incomparable) metrizable topologies $\tau_0$, $\tau_1$ such that the Kuratowski monoid $\IK(\mathbf X)$ of the multitopological space $\mathbf X=(X,\{\tau_0,\tau_1\})$ is infinite. Moreover, for some set $A\subset X$ the family $\{f(A):f\in\IK(\mathbf X)\}$ is infinite.
\end{example}

\begin{proof} Take any countable metrizable topological space $X$ containing a decreasing sequence of non-empty subsets $(X_n)_{n\in\w}$ such that $X_0=X$, $\bigcap_{n\in\w}X_n=\emptyset$ and $X_{n+1}$ is nowhere dense in $X_n$ for all $n\in\w$.

To find such a space $X$, take the convergent sequence $S_0=\{0\}\cup\{2^{-n}:n\in\w\}$ and consider the subspace $$X=\{(x_k)_{k\in\w}\in S_0^\w:\exists n\in\w\;\forall k\ge n\;\;x_k=0\}\setminus\{0\}^\w$$of the countable power $S_0^\w$ endowed with the Tychonoff product topology. It is easy to see that the subsets $$X_{n}=\{(x_k)_{k\in\w}\in X:\forall k<n\;\;x_k=0\},\;\;n\in\w,$$ have the required properties: $X_0=X$, $\bigcap_{n\in\w}X_n=\emptyset$ and $X_{n+1}$ is nowhere dense in $X_n$.

On the space $X$ consider two topologies
$$\tau_0=\{U\subset X:\forall n\in\w\;\;\mbox{$U\cap (X_{2n}\setminus X_{2n+2})$ is open in $X_{2n}\setminus X_{2n+2}$}\}$$and
$$\tau_1=\{U\subset X:\forall n\in\w\;\;\mbox{$U\cap (X_{2n+1}\setminus X_{2n+3})$ is open in $X_{2n+1}\setminus X_{2n+3}$}\}.$$
Observe that $\tau_0$ coincides with the topology of the topological sum $\bigoplus_{n\in\w}X_{2n}\setminus X_{2n+2}$ while $\tau_2$ coincides with the topology of the topological sum $\big(\bigoplus_{x\in X_0\setminus X_1}\{x\}\big)\oplus\bigoplus_{n\in\w}X_{2n+1}\setminus X_{2n+3}$.

We claim that the multitopological space $\mathbf X=(X,\{\tau_1,\tau_2\})$ has the required properties.

By $\ctau_1,\ctau_2:\mathcal P(X)\to\mathcal P(X)$ we denote the closure operators in the topologies $\tau_1$ and $\tau_2$, respectively. The nowhere density of $X_{n+1}$ in $X_n$ for all $n\in\w$ and the definition of the topologies $\tau_0,\tau_1$ imply that for every $n\in\w$
\begin{enumerate}
\item $\ctau_0(X\setminus X_{2n+1})=X\setminus X_{2n+2}$;
\item $\ctau_1(X\setminus X_{2n+2})=X\setminus X_{2n+3}$,
\item $\ctau_1\ctau_0(X\setminus X_{2n+1})=X\setminus X_{2n+3}$,
\item $(\ctau_1\ctau_0)^n(X\setminus X_1)=X\setminus X_{2n+1}$.
\end{enumerate}

Therefore, for the set $A=X\setminus X_1$ the sets $(\ctau_1\ctau_0)^n(A)$, $n\in\w$, are pairwise distinct, which implies that the family $\{f(A):f\in\IK(\mathbf X)\}$ is infinite and hence the Kuratowski monoid $\IK(\mathbf X)$ of the multitopological space $\mathbf X=(X,\{\tau_0,\tau_1\})$ is infinite too.
\end{proof}

\section{Kuratowski monoids}

To prove Theorem~\ref{main} we shall use the natural structure of partial order on the monoid $\mathcal P(X)^{\mathcal P(X)}$. For two maps $f,g\in \mathcal P(X)^{\mathcal P(X)}$ we write $f\le g$ if $f(A)\subset g(A)$ for every subset $A\subset X$. This partial order turns $\mathcal P(X)$ into a partially ordered monoid.

By a {\em partially ordered monoid\/} we understand a monoid $M$ endowed with a partial order $\le$ which is {\em compatible} with the semigroup operation of $M$ in the sense that for any points $x,y,z\in M$ the inequality $x\le y$ implies $xz\le yz$ and $zx\le zy$. Recall that a {\em monoid\/} is a semigroup $S$ possessing a two-sided unit $1\in S$.

Observe that for two comparable topologies $\tau_1\subset\tau_2$ on a set $X$ we get
$$\itau_1\le\itau_2\le 1_X\le\ctau_2\le\ctau_1$$
where $1_X:\mathcal P(X)\to\mathcal P(X)$ is the identity transformation of $\mathcal P(X)$. Now we see that for a polytopological space $\mathbf X=(X,\Tau)$ its Kuratowski monoid $\IK(\mathbf X)=\la \itau,\ctau:\tau\in\Tau\ra$ is generated by the linearly ordered set $$L(\mathbf X)=\{\itau:\tau\in\Tau\}\cup\{1_X\}\cup\{\ctau:\tau\in\Tau\}.$$ This leads to the following

\begin{definition} A {\em Kuratowski monoid\/} is a partially ordered monoid $K$ generated by a finite linearly ordered set $L$ containing the unit $1$ of $K$ and consisting of idempotents.

The set $L$ will be called a {\em linear generating set} of the Kuratowski monoid $K$. This set can be written as the union $L=L_{-}\cup \{1\}\cup L_{+}$ where $L_{-}=\{x\in L:x<1\}$ and $L_{+}=\{x\in L:x>1\}$ are the sets of {\em negative} and {\em positive generating elements} of $K$.

A Kuratowski monoid $K$ is called a {\em Kuratowski monoid of type $(n,p)$} if $K$ has a linear generating set $L$ such that $|L_{-}|=n$ and $|L_{+}|=p$.
\end{definition}

For two numbers $n,p\in\w$ consider the number
$$K(n,p)=\sum_{i=0}^n\sum_{j=0}^p\textstyle{\binom{i+j}{i}\cdot\binom{i+j}{j}}$$and observe that $K(n)=K(n,n)$ for every $n\in\w$.

It is easy to see that for each polytopological space $\mathbf X=(X,\Tau)$ endowed with $n=|\Tau|$ topologies, its Kuratowski monoid $\mbox{K}(\mathbf X)$ is a Kuratowski monoid of type $(n,n)$ or $(n-1,n-1)$. The latter case happens if the polytopology $\Tau$ of $\mathbf X$ contains the discrete topology $\tau_d$ on $X$. In this case $\breve\tau_d=1_X=\bar\tau_d$.

 Now we see that Theorem~\ref{main} is a partial case of the following more general theorem, which will be proved in Section~\ref{pf:kur} (more precisely, in Theorem~\ref{represent}).

\begin{theorem}\label{main(n,p)} Each Kuratowski monoid $K$ of type $(n,p)$ has cardinality $|K|\le K(n,p)$.
\end{theorem}

The values of the double sequence $K(n,p)$ for $n,p\le 9$ were calculated by computer:
\medskip

\begin{tabular}{|l|rrrrrrrrrr}
\hline
$n\backslash p$
  &0      &    1  &    2   & 3       & 4     & 5     & 6      & 7    & 8   & 9\\
\hline
0 &{\bf 1}& 2     & 3      & 4       & 5     & 6     & 7      &8&9&10 \\
1 & 2     &{\bf7} & 17     & 34      & 60    & 97    & 147    & 212  & 294  & 395\\
2 & 3     & 17    &{\bf 63}& 180     & 431  & 909   &1743    & 3104 & 5211 &8337\\
3 & 4     & 34    & 180    &{\bf 697}& 2173 & 5787 &13677   & 29438 & 58770&110296\\
4 & 5     & 60    & 431    & 2173    &{\bf8549}&28039&80029 &204690 &479047&1041798\\
5 & 6     & 97    & 909    & 5787    & 28039   &{\bf111033}&376467 &1128392& 3059118&7629873\\
6 & 7     & 147   &1743    &13677    &80029    &376467     &{\bf1495677}&5192258&16140993&45761773 \\
7 & 8     & 212   &3104    &29438    &204690   &1128392    &5192258     &{\bf20667463}&73025423&233519803\\
8 & 9     & 294   &5211    &58770    &479047   &3059118    &16140993    &73025423&{\bf291020283} &1042490763\\
9 &10     &395    &8337    &110296   &1041798  &7629873    &45761773   &233519803&1042490763&{\bf 4157865643}\\
\end{tabular}

\section{Kuratowski words}

Theorem~\ref{main(n,p)} will be proved by showing that each element of a Kuratowski monoid $K$ with linear generating set $L$ can be represented by a Kuratowski word in the alphabet $L$.
Kuratowski words are defined as follows.

By a {\em pointed linearly ordered set} we understand a linearly ordered set $L$ with a distinguished element $1\in L$ called the {\em unit} of $L$. This element divides the set $L\setminus \{1\}$ into negative and positive parts  $L_-=\{x\in L:x<1\}$ and $L_+=\{x\in L:x>1\}$, respectively. By $FS_L=\bigcup_{n=1}^\infty L^n$ we denote the free semigroup over $L$. It consists of non-empty words in the alphabet $L$. The semigroup operation on $FS_L$ is defined as the concatenation of words. The set $L$ is identified with the set $L^1$ of words of length $1$ in the alphabet $L$.

A word $w=x_1\dots x_n\in FS_L$ of length $n$ is called {\em alternating} if for each natural number $i$ with $1\le i<n$ the doubleton $\{x_i,x_{i+1}\}$ intersects both sets $L_{_-}$ and $L_+$. According to this definition, words of length 1 also are alternating. On the other hand, an alternating word of length $\ge 2$ does not contain a letter equal to $1$.

An alternating word $x_0\cdots x_n\in FS_L$ of length $n+1\ge 2$ is defined to be
\begin{itemize}
\item a {\em $V_\mp$-word} iff there is an integer number $m\in\{0,\dots,n-1\}$ such that the sequences $(x_{m+2i})_{0{\le} i{\le} \frac{n-m}2}$, $(x_{m-2i})_{0{\le}i\le\frac{m}2}$ are strictly increasing in $L_-$ and the sequences $(x_{m+1+2i})_{0{\le}i{<}\frac{n-m}2}$, $(x_{m+1-2i})_{0{\le}i\le\frac{m+1}2}$ are strictly decreasing in $L_+$;
\item a {\em $V_\pm$-word} if there is a number $m$ such that $m\in\{0,\dots,n-1\}$ such that the sequences $(x_{m+2i})_{0{\le}i{\le}\frac{n-m}2}$, $(x_{m-2i})_{0{\le}i{\le}\frac{m}2}$ are strictly decreasing in $L_+$ and the sequences $(x_{m+1+2i})_{0{\le}i{<}\frac{n-m}2}$, $(x_{m+1-2i})_{0{\le}i\le\frac{m+1}2}$ are strictly increasing in $L_-$;
\item a {\em $W_-$-word} if there is a number $m\in\{1,\dots,n-1\}$ such that $x_{m-1}=x_{m+1}\in L_-$, the sequences $(x_{m+1+2i})_{0{\le}i{<}\frac{n-m}2}$, $(x_{m-1-2i})_{0{\le}i{\le}\frac{m-1}2}$ are strictly increasing in $L_-$ and the sequences $(x_{m+2i})_{0{\le}i{\le}\frac{n-m}2}$, $(x_{m-2i})_{0{\le}i{\le}\frac{m}2}$ are strictly decreasing in $L_+$;
\item a {\em $W_+$-word} if there is a number $m\in\{1,\dots,n-1\}$ such that $x_{m-1}=x_{m+1}\in L_+$, the sequences $(x_{m+1+2i})_{0{\le}i{<}\frac{n-m}2}$, $(x_{m-1-2i})_{0{\le}i\le\frac{m-1}2}$ are strictly decreasing  in $L_+$ and the sequences $(x_{m+2i})_{0{\le}i{\le}\frac{n-m}2}$, $(x_{m-2i})_{0{\le}i\le\frac{m}2}$ are strictly increasing  in $L_-$.
\end{itemize}

By $\V_\mp$ (resp. $\V_\pm$, $\W_-$, $\W_+$) we denote the family of all $V_\mp$-words (resp. $V_\pm$-words, $W_-$-word, $W_+$-words) in the alphabet $L$. It is easy to see that
the families of words $\V_\mp$, $\V_\pm$, $\W_-$, $\W_+$ are pairwise disjoint.
Words that belong to the set $$\K_L=L^1\cup \V_\pm\cup\V_\mp\cup\W_-\cup\W_+$$are called {\em Kuratowski words} in the alphabet $L$.

Let us calculate the cardinality of the set $\K_L$ depending on the cardinalities $n=|L_-|$ and $p=|L_+|$ of the negative and positive parts of $L$.

For non-negative integers $n,r$ by $\binom{n}{r}$ we denote the cardinality of the set of $r$-element subsets of an $n$-element set. It is clear that
$${\textstyle\binom{n}{r}}=\begin{cases}
\frac{n!}{r!(n-r)!}&\mbox{if $0\le r\le n$};\\
0&\mbox{otherwise}.
\end{cases}
$$
The numbers $\binom{n}{r}$ will be called binomial coefficients. The following properties of binomial coefficients are well-known (see, e.g. \cite[\S 5.1]{CM}).

\begin{lemma}\label{lb1} For any non-negative integer numbers $m,n,k$ we get
\begin{enumerate}
\item $\binom{n}{k}=\binom{n}{n-k}$,
\item $\binom{n}{k}+\binom{n}{k-1}=\binom{n+1}{k}$, and
\item $\binom{n+m}{k}=\sum_{l=0}^n\binom{n}{l}\binom{m}{k-l}.$
\end{enumerate}
\end{lemma}

In the following theorem we calculate the cardinality of the set $\K_L$ of Kuratowski words.

\begin{theorem}\label{kuratwords} For any finite pointed linearly ordered set $L$ with $n=|L_-|$ and $p=|L_+|$ we get
\begin{enumerate}
\item $|\V_\mp|= \sum\limits_{a=0}^{n-1}\sum\limits_{b=0}^{p-1}\sum\limits_{l=0}^{a}\sum\limits_{r=0}^{a}\binom{a}{l}\binom{a}{r} \binom{b+1}{l+1}\binom{b+1}{r}$,
\item $|\V_\pm|= \sum\limits_{a=0}^{n-1}\sum\limits_{b=0}^{p-1}\sum\limits_{l=0}^{a}\sum\limits_{r=0}^{a} \binom{a}{l}\binom{a}{r}\binom{b+1}{l}\binom{b+1}{r+1}$,
\item $|\W_+|= \sum\limits_{a=0}^{n-1}\sum\limits_{b=0}^{p-1}\sum\limits_{l=0}^{a}\sum\limits_{r=0}^{a} \binom{a}{l}\binom{a}{r}\binom{b+1}{l}\binom{b+1}{r}$,
\item $|\W_-|= \sum\limits_{a=0}^{n-1}\sum\limits_{b=0}^{p-1}\sum\limits_{l=0}^{a}\sum\limits_{r=0}^{a}
    \binom{a}{l}\binom{a}{r}\binom{b+1}{l+1}\binom{b+1}{r+1}$,
\item $|\K_L|=\sum_{a=0}^n\sum_{b=0}^p\textstyle\binom{a+b}{a}\binom{a+b}{b}=K(n,p)$.
\end{enumerate}
\end{theorem}

\begin{proof} 1. To calculate the number of $V_\mp$-words, fix any $V_\mp$-word $v\in V_{\mp}$ and write it as an alternating word $v=x_k\dots x_{2m}\dots x_q$ such that the sequences $(x_{2m+2i})_{0{\le}i{\le}\frac{q-2m}2}$ and $(x_{2m-2i})_{0{\le}i{\le}\frac{2m-k}2}$ are strictly increasing in $L_-$ and the sequences $(x_{2m+1+2i})_{0{\le}i{\le}\frac{q-2m-1}2}$ and $(x_{2m+1-2i})_{0{\le}i\le \frac{2m+1-k}2}$ are strictly decreasing in $L_+$. It follows that the sequence $(x_{2m+2i})_{1{\le}i{\le}\frac{q-2m}2}$ is a strictly increasing sequence of length $r=\lfloor\frac{q-2m}2\rfloor$ in the linearly ordered set $A=\{x\in L:x>x_{2m}\}\subset L_-$. The number of such sequences is equal to $\binom{a}{r}$ where the cardinality $a=|A|$ can vary from $0$ (if $x_{2m}$ is the largest element of the set $L_-$) till $n-1$ (if $x_{2m}$ is the smallest element of $L_-$). By analogy, $(x_{2m-2i})_{1{\le}i{\le}\frac{2m-k}2}$ is a strictly increasing sequence of length $l=\lfloor\frac{2m-k}2\rfloor$ in the linearly ordered set $A$ and the number of such sequences is equal to $\binom{a}{l}$.

If $l=\frac{2m-k}2$, then $2m=k+2l$ and $(x_{2m+1-2i})_{1{\le}i\le \frac{2m+1-k}2}$ is a strictly decreasing sequence of length $\lfloor\frac{2m+1-k}2\rfloor= l$ in the linearly ordered set $B=\{x\in L:x<x_{2m+1}\}\subset L_+$. The number of such sequences is equal to $\binom{b}{l}$ where $b=|B|<p$. If $l=\frac{2m-k-1}2$, then $2m=k+1+2l$ and $(x_{2m+1-2i})_{1{\le}i\le \frac{2m+1-k}2}$ is a strictly increasing sequence of length $\lfloor\frac{2m+1-k}2\rfloor=l+1$ in the set $B$. The number of such sequences is equal to $\binom{b}{l+1}$.

If $r=\frac{q-2m}2$, then $2m=q-2r$ and $(x_{2m+1+2i})_{1{\le}i{\le}\frac{q-2m-1}2}$ is a strictly decreasing sequence of length $\lfloor\frac{q-2m-1}2\rfloor= r-1$ in the linearly ordered set $B$. The number of such sequences is equal to $\binom{b}{r-1}$. If $r=\frac{q-2m-1}2$, then $2m=q-1-2r$ and $(x_{2m+1+2i})_{1{\le}i{\le}\frac{q-2m-1}2}$ is a strictly decreasing sequences of length $\lfloor \frac{q-2m-1}2\rfloor=r$ in the linearly ordered set $B$. The number of such sequences is equal to $\binom{b}{r}$. Summing up and applying Lemma~\ref{lb1}(2), we conclude that the family $\V_\mp$ of all $V_\mp$-words has cardinality
$$
\begin{aligned}
|\V_\mp|&= \sum_{a=0}^{n-1}\sum_{l=0}^{a}\sum_{r=0}^{a}\sum_{b=0}^{p-1}{\textstyle\binom{a}{l}\binom{a}{r}\Big(
    \binom{b}{l}\binom{b}{r}+\binom{b}{l+1}\binom{b}{r}+\binom{b}{l}\binom{b}{r-1}+\binom{b}{l+1}
    \binom{b}{r-1}\Big)}=\\
&=\sum_{a=0}^{n-1}\sum_{b=0}^{p-1}\sum_{l=0}^{a}\sum_{r=0}^{a}{\textstyle\binom{a}{l}\binom{a}{r}
\Big(\binom{b+1}{l+1}\binom{b}{r}+\binom{b+1}{l+1}\binom{b}{r-1}\Big)}=
\sum_{a=0}^{n-1}\sum_{b=0}^{p-1}\sum_{l=0}^{a}\sum_{r=0}^{a}{\textstyle\binom{a}{l}\binom{a}{r}
\binom{b+1}{l+1}\binom{b+1}{r}}.
\end{aligned}
$$
\smallskip

2. By analogy we can prove that
$$|\V_\pm|= \sum_{a=0}^{n-1}\sum_{b=0}^{p-1}\sum_{l=0}^{a}\sum_{r=0}^{a}\textstyle\binom{a}{l}\binom{a}{r}
    \binom{b+1}{l}\binom{b+1}{r+1}.$$
\smallskip

3. To calculate the number of $W_+$-words, fix any $W_+$-word $w\in \W_{+}$ and write it as an alternating word $w=x_k\dots x_{2m}\dots x_q$ such that $k<2m<q$,  $x_{2m}\in L_-$, $x_{2m-1}=x_{2m+1}\in L_+$, the sequences $(x_{2m+2i})_{0{\le}i{\le}\frac{q-2m}2}$ and
$(x_{2m-2i})_{0{\le}i{\le}\frac{2m-k}2}$ are strictly increasing in $L_-$ whereas the sequences
$(x_{2m+1+2i})_{0{\le}i{\le}\frac{q-2m-1}2}$ and\break $(x_{2m-1-2i})_{0{\le}i{\le}\frac{2m-k-1}2}$ are strictly decreasing in $L_+$.

 It follows that $(x_{2m-2i})_{1{\le}i{\le}\frac{2m-k}2}$ is a strictly increasing sequence of length $l=\lfloor\frac{2m-k}2\rfloor$ in the linearly ordered set $A=\{x\in L:x>x_{2m}\}\subset L_-$. The number of such sequences is equal to $\binom{a}{l}$ where  $a=|A|<n$. By analogy, $(x_{2m+2i})_{0{\le}i{\le}\frac{q-2m}2}$ is a strictly increasing sequence of length $r=\lfloor \frac{q-2m}2\rfloor$ in the linearly ordered set $A$ and the number of such sequences is equal to $\binom{a}{r}$.

If $l=\frac{2m-k}2$, then $2m=2l+k$ and $(x_{2m-1-2i})_{1{\le}i{\le}\frac{2m-k-1}2}$ is a strictly decreasing sequence of length $\lfloor \frac{2m-k-1}2\rfloor =\lfloor\frac{2l-1}2\rfloor=l-1$ in the linearly ordered set $B=\{x\in L:x<x_{2m-1}=x_{2m+1}\}\subset L_+$. The number of such sequences is equal to $\binom{b}{l-1}$ where $b=|B|<p$. If $l=\frac{2m-k-1}2$, then $(x_{2m-1-2i})_{1{\le}i{\le}\frac{2m-k-1}2}$ is a strictly decreasing sequence of length $\lfloor \frac{2m-k-1}2\rfloor=l$ in the set $B$. The number of such sequences is equal to $\binom{b}{l}$.

If $r=\frac{q-2m}2$, then $2m=q-2r$ and
 $(x_{2m+1+2i})_{0{\le}i{\le}\frac{q-2m-1}2}$ is a strictly decreasing sequence of length $\lfloor\frac{q-2m-1}2\rfloor=r-1$ in the linearly ordered set $B$. The number of such sequences is equal to $\binom{b}{r-1}$. If $r=\frac{q-2m-1}2$, then $2m=q-1-2r$ and
  $(x_{2m+1+2i})_{1{\le}i{\le}\frac{q-2m-1}2}$ is a strictly decreasing sequence of length $\lfloor\frac{q-2m-1}2\rfloor=r$ in the linearly ordered set $B$.
The number of such sequences is equal to $\binom{b}{r}$. Summing up, we conclude that the family $\W_+$ of all $W_+$-words has cardinality
$$
\begin{aligned}
|\W_+|&= \sum_{a=0}^{n-1}\sum_{l=0}^{a}\sum\limits_{r=0}^{a}\sum_{b=0}^{p-1}{\textstyle\binom{a}{l}\binom{a}{r}\Big(
    \binom{b}{l}\binom{b}{r}+\binom{b}{l}\binom{b}{r-1}+\binom{b}{l-1}\binom{b}{r}+\binom{b}{l-1}
    \binom{b}{r-1}\Big)}=\\
&=\sum_{a=0}^{n-1}\sum_{b=0}^{p-1}\sum_{l=0}^{a}\sum\limits_{r=0}^{a}
{\textstyle \binom{a}{l}\binom{a}{r}\Big(\binom{b}{l}\binom{b+1}{r}+\binom{b}{l-1}\binom{b+1}{r}\Big)}=
    \sum_{a=0}^{n-1}\sum_{b=0}^{p-1}\sum_{l=0}^{a}\sum_{r=0}^{a}\textstyle{\binom{a}{l}\binom{a}{r}
    \binom{b+1}{l}\binom{b+1}{r}}.
\end{aligned}
$$
\smallskip

4. By analogy we can prove that
$$|\W_-|= \sum_{a=0}^{n-1}\sum_{b=0}^{p-1}\sum_{l=0}^{a}\sum_{r=0}^{a}\textstyle\binom{a}{l}\binom{a}{r}
    \binom{b+1}{l+1}\binom{b+1}{r+1}.$$
\smallskip

5. By the preceding items
$$
\begin{aligned}
|\K_L|&=|L|+|\V_\mp|+|\V_\pm|+|\W_-|+|\W_+|=1+n+p+|\V_\mp|+|\W_-|+|\V_\pm|+|\W_+|=\\
&= 1+n+p+\sum_{a=0}^{n-1}\sum_{b=0}^{p-1}\sum_{l=0}^{a}\sum_{r=0}^{a}{\textstyle\binom{a}{l}\binom{a}{r} \Big(\binom{b+1}{l+1}\binom{b+1}{r}+\binom{b+1}{l}\binom{b+1}{r}+\binom{b+1}{l}\binom{b+1}{r+1}+
\binom{b+1}{l+1}\binom{b+1}{r+1}\Big)}=\\
&=1+n+p+\sum_{a=0}^{n-1}\sum_{b=0}^{p-1}\sum_{l=0}^{a}\sum_{r=0}^{a}{\textstyle \binom{a}{l}\binom{a}{r} \Big(\binom{b+2}{l+1}\binom{b+1}{r}+\binom{b+2}{l+1}\binom{b+1}{r+1}\Big)}=\\
&=1+n+p+\sum_{a=0}^{n-1}\sum_{b=0}^{p-1}\sum_{l=0}^{a}\sum_{r=0}^{a}{\textstyle \binom{a}{l}\binom{a}{r} \binom{b+2}{l+1}\binom{b+2}{r+1}}=1+n+p+\sum_{a=0}^{n-1}\sum_{b=0}^{p-1}\Big(\sum_{l=0}^{a}
{\textstyle\binom{a}{l}\binom{b+2}{l+1}}\Big)^2=\\
&=1+n+p+\sum_{a=0}^{n-1}\sum_{b=0}^{p-1}\Big(\sum_{l=0}^{a}{\textstyle\binom{a}{l}\binom{b+2}{b+1-l}}\Big)^2=1+n+p+
\sum_{a=0}^{n-1}\sum_{b=0}^{p-1}{\textstyle\binom{a+b+2}{b+1}}^2=\\
&=1+n+p+\sum_{a=1}^{n}\sum_{b=1}^{p}
{\textstyle\binom{a+b}{b}}^2
=\sum_{a=0}^{n}\sum_{b=0}^{p}\textstyle{\binom{a+b}{b}\binom{a+b}{a}}=K(n,p).
\end{aligned}
$$
\end{proof}

\section{An asymptotics of the sequence $K(n)$}\label{s:asymp}

In this section we study the asymptotical growth of the sequence $K(n)=K(n,n)$ and prove Theorem~\ref{asymp} announced in the Introduction
as a corollary of the following results.

For every integers $0\le a,b\le n$ put
$$c_{a,b}(n)=\frac{\binom {2n-a-b}{n-a}}{\binom {2n}n}=\frac {n(n-1)\cdots (n-a+1)\cdot n(n-1)\cdots (n-b+1)}{2n(2n-1)\cdots (2n-a-b+1)}$$ and observe that
\begin{equation}\label{eq:my}
\lim_{n\to\infty} c_{a,b}(n)=2^{-(a+b)}.
\end{equation}

For every $n\ge 0$ put $k(n)=K(n)/\binom{2n}{n}^2$ and observe that
$$k(n)=\frac{K(n)}{\binom {2n}n^2}=\sum_{i,j=0}^n\frac{\binom{i+j}{i}^2}{\binom{2n}{n}^2}=
\sum_{a,b=0}^n\frac{\binom{n-a+n-b}{n-a}^2}{\binom{2n}{n}^2}=\sum_{a,b=0}^nc_{a,b}(n)^2.$$
%Computer calculations suggest that the sequence $\{k(n):n\ge 3\}$
%is increasing.

\begin{proposition}\label{ravp1} For every $n\ge0$ we have $k(n)\le \frac{16}{9}$.
\end{proposition}
\begin{proof} $k(0)=1$, $k(1)=k(2)=\frac{7}{4}<\frac{16}{9}$, $k(3)=\frac{697}{400}<\frac{7}{4}$, and $k(4)=\frac{8549}{4900}<\frac{16}{9}$. Suppose now that for some $n\ge 5$ we have proved that $k(n-1)\le \frac{16}{9}$.

We shall use the following two lemmas.

\begin{lemma} For each $0<a\le n$ we have $c_{a,0}(n)<2^{-a}$.
\end{lemma}
\begin{proof} This lemma follows from the equality $c_{a,0}(n)=\frac {n(n-1)\cdots (n-a+1)}{2n(2n-1)\cdots (2n-a+1)}$ and the inequality $(2n-l)> 2(n-l)$ holding for all $0<l\le n$.
\end{proof}

\begin{lemma}\label{l:6.3} $\frac{16}{9}c_{1,1}(n)^2+2\big(c_{1,0}(n)^2+c_{2,0}(n)^2\big)< \frac{1}{9}+2\big(\frac1{4}+\frac1{4^2}\big)$ for $n\ge 5$.
\end{lemma}
\begin{proof} Observe that $c_{1,1}(n)=\frac n{2(2n-1)}$, $c_{1,0}(n)=1/2$, and $c_{2,0}(n)=\frac {n-1}{2(2n-1)}$. Routine transformations show that the inequality in the lemma are equivalent to $19< 4n$, which holds for $n\ge 5$.
\end{proof}

Now we have that
$$
\begin{aligned}
k(n)&=c_{0,0}(n)^2+2\sum_{a=1}^nc_{a,0}(n)^2+\sum_{a,b=1}^nc_{a,b}(n)^2=1+2\sum_{a=1}^nc_{a,0}(n)^2+\sum_{a,b=0}^{n-1}c_{a+1,b+1}(n)^2=\\
&=1+2\sum_{a=1}^nc_{a,0}(n)^2+\sum_{a,b=0}^{n-1}\tfrac{n^4}{(2n(2n-1))^2}c_{a,b}(n-1)^2=\\
&=1+2\sum_{a=3}^n c_{a,0}(n)^2+2\sum_{a=1}^2c_{a,0}(n)^2+c_{1,1}(n)^2k(n-1)<\\
&<1+2\sum_{a=3}^n\frac1{4^a}+2(c_{1,0}(n)^2+c_{2,0}(n)^2)+c_{1,1}(n)^2\cdot\frac{16}{9}<\\
&<1+2\sum_{a=3}^\infty\frac1{4^a}+2\Big(\frac14+\frac1{4^2}\Big)+\frac{1}{9}=\frac{16}9
\end{aligned}
$$
according to Lemma~\ref{l:6.3}.

\end{proof}

\begin{proposition}\label{ravp2} There exists a limit $\lim_{n\to\infty} k(n)=16/9$.
\end{proposition}
\begin{proof} The equality (\ref{eq:my}) implies that $\underline\lim_{n\to\infty} k(n)\ge \sum_{a,b=0}^\infty 2^{-2(a+b)}
=\sum_{i=0}^\infty (i+1)4^{-i}=16/9$. By Proposition~\ref{ravp1}, $\lim_{n\to\infty} k(n)=16/9$.
\end{proof}

By Stirling's approximation, $\lim_{n\to\infty}\frac {n!}{\sqrt{2\pi n}}(n/e)^n=1$, which yields
the second equality in Theorem~\ref{asymp}.

\section{Representing elements of Kuratowski monoids by Kuratowski words}\label{pf:kur}

Let $K$ be a Kuratowski monoid with linear generating set $L$ and let $L_-=\{x\in L:x<1\}$ and $L_+=\{x\in L:x>1\}$ be the negative and positive parts of $L$, respectively. Let $FS_L=\bigcup_{n=1}^\infty L^n$ be the free semigroup over $L$ and $\pi:FS_L\to K$ be the homomorphism assigning to each word $x_1\dots x_n\in FS_L$ the product $x_1\cdots x_n$ of its letters in $K$. The homomorphism $\pi:FS_L\to K$ induces a congruence $\sim$ on $FS_L$ which identifies two words $u,v\in FS_L$ iff $\pi(u)=\pi(v)$.

A word $w\in FS_L$ is called {\em irreducible} if $w$ has the smallest possible length in its equivalence class $[w]_\sim=\{u\in FS_L:u\sim w\}$. Since the set of natural numbers is well-ordered, for each element $x\in K$ there is an irreducible word $w\in FS_L$ such that $x=\pi(w)$. Consequently, the cardinality of $K$ does not exceed the cardinality of the set of irreducible words in $FS_L$.

\begin{theorem}\label{represent} Each irreducible word in $FS_L$ is a Kuratowski word. Consequently, $\pi(\K_L)=K$ and $|K|\le|\K_L|$. If the set $L$ is finite, then $|K|\le|\K_L|=K(n,p)$ where $n=|L_-|$ and $p=|L_+|$.
\end{theorem}

\begin{proof} We divide the proof of Theorem~\ref{represent} into a series of lemmas.

\begin{lemma}\label{l1} For any elements $x,y\in L_-\cup\{1\}$ we get $xy=\min\{x,y\}$.
\end{lemma}

\begin{proof}
Since $L$ is linearly ordered, either $x\le y$ or $y\le x$.

If $x\le y$, then multiplying this inequality by $x$, we get $x=xx\le xy$. On the other hand, multiplying the inequality $y\le 1$ by $x$, we get the reverse inequality $xy\le x1=x$. Taking into account that $x\le xy\le x$, we conclude that $x=xy$.

If $y\le x$, then multiplying this inequality by $y$, we get $y=yy\le xy$. On the other hand, multiplying the inequality $x\le 1$ by $y$, we get $xy\le 1y=y$. Taking into account that $y\le xy\le y$, we conclude that $xy=y=\min\{x,y\}$.
\end{proof}

By analogy we can prove:

\begin{lemma}\label{l2} For any elements $x,y\in L_+\cup\{1\}$ we get $xy=\max\{x,y\}$.
\end{lemma}

\begin{proof} If $x\le y$, then multiplying this inequality by $y$, we obtain  $xy\le yy=y$. On the other hand, multiplying the inequality $1\le x$ by $y$, we get $y=1y\le xy$. So, $xy=y=\max\{x,y\}$.

If $y\le x$, then after multiplication by $x$, we obtain $xy\le xx=x$. On the other hand, multiplying the inequality $1\le y$ by $x$, we get $x\le xy$ and hence $xy=x=\max\{x,y\}$.
\end{proof}

Recall that a word $x_1\dots x_n\in FS_L$ is {\em alternating} if for each natural number $i$ with $1\le i<n$ the doubleton $\{x_i,x_{i+1}\}$ intersects both sets $L_{_-}$ and $L_+$. According to this definition, one-letter words also are alternating. Lemmas~\ref{l1} and \ref{l2} imply:

\begin{lemma}\label{l3} Each irreducible word $w\in FS_L$ is alternating. \end{lemma}

The following lemma will help us to reduce certain alternating words of length 4.

\begin{lemma}\label{l4} If $x_1x_2x_3x_4\in FS_L$ is an alternating word in the alphabet $L$ such that $x_1x_3=x_1$ and $x_2x_4=x_4$ in $K$, then $x_1x_2x_3x_4=x_1x_4$ in $K$ and hence $x_1x_2x_3x_4\sim x_1x_4$.
\end{lemma}

\begin{proof} Two cases are possible.
\smallskip

1) $x_1,x_3\in L_-$ and $x_2,x_4\in L_+$. In this case the equalities $x_1x_3=x_1$ and $x_2x_4=x_4$ imply that $x_1\le x_3$ and $x_2\le x_4$ (see Lemmas~\ref{l1} and \ref{l2}). To see that $x_1x_2x_3x_4= x_1x_4$, observe that
$$x_1x_2x_3x_4\le x_1x_2\cdot 1\cdot x_4=x_1x_2x_4=x_1x_4.$$
On the other hand,
$$x_1x_4=(x_1x_3)x_4=x_1\cdot 1\cdot x_3x_4\le x_1x_2x_3x_4.$$These two inequalities imply the desired equality $x_1x_2x_3x_4=x_1x_4$.
\smallskip

2) $x_1,x_3\in L_+$ and $x_2,x_4\in L_-$.  In this case the equalities $x_1x_3=x_1$ and $x_2x_4=x_4$ imply that $x_1\ge x_3$ and $x_2\ge x_4$ (see Lemmas~\ref{l1} and \ref{l2}). To see that $x_1x_2x_3x_4=x_1x_4$, observe that
$$x_1x_2x_3x_4\le x_1\cdot 1\cdot x_3x_4=x_1x_3x_4=x_1x_4.$$
On the other hand,
$$x_1x_4=x_1(x_2x_4)=x_1x_2\cdot 1\cdot x_4\le x_1x_2x_3x_4.$$
These two inequalities imply the desired equality $x_1x_2x_3x_4=x_1x_4$.
\end{proof}

%An alternating word $x_0\cdots x_n$ in the alphabet $L\setminus\{1\}$ will be called
%\begin{itemize}
%\item a {\em $V_\mp$-word} if there is an integer number $m$ such that $0\le m<n$, $x_m\in L_-$, $x_{m+1}\in L_+$, $x_{i-1}x_{i+1}=x_{i+1}\ne x_{i-1}$ for any $0<i\le m$ and $x_{j-1}x_{j+1}=x_{j-1}\ne x_{j+1}$ for any number $j$ with $m<j<n$;
%\item a {\em $V_\pm$-word} if there is a number $m$ such that $0\le m<n$ such that $x_m\in L_+$, $x_{m+1}\in L_-$, $x_{i-1}x_{i+1}=x_{i+1}\ne x_{i-1}$ for any $0<i\le m$ and $x_{j-1}x_{j+1}=x_{j-1}\ne x_{j+1}$ for any $m<j<n$;
%\item a {\em $W_-$-word} if there is a number $m$ such that $0<m<n$ such that $x_m\in L_+$,  $x_{m-1}=x_{m+1}\in L_-$, $x_{i-1}x_{i+1}=x_{i+1}\ne x_{i-1}$ for any $0<i<m$ and $x_{j-1}x_{j+1}=x_{j-1}\ne x_{j+1}$ for any $m<j<n$;
%\item a {\em $W_+$-word} if there is a number $m$ such that $0<m<n$ such that $x_m\in L_-$,  $x_{m-1}=x_{m+1}\in L_+$, $x_{i-1}x_{i+1}=x_{i+1}\ne x_{i-1}$ for any $0<i<m$ and $x_{j-1}x_{j+1}=x_{j-1}\ne x_{j+1}$ for any $m<j<n$;
%\end{itemize}

Now we are able to prove that each irreducible word $w\in FS_L$ is a Kuratowski word.
If $w$ consists of a single letter, then it is trivially Kuratowski and we are done. So, we assume that $w$ has length $\ge 2$. By Lemma~\ref{l3} the word $w$ is alternating and hence can be written as the product $w=x_k\cdots x_n$ for some $k\in\{0,1\}$ and $n>k$ such that $x_{2i}\in L_-$ for all integer numbers $i$ with $k\le 2i\le n$, and $x_{2i-1}\in L_+$ for all integer numbers $i$ with $k\le 2i-1\le n$.

Let $m$ be the smallest number such that $k\le 2m\le n$ and $x_{2m}=\min\{x_{2i}:k\le 2i\le n\}$ in $L_-$. First we shall analyze the structure of the subword $x_k\cdots x_{2m}$ of the word $w=x_k\dots x_n$.

\begin{lemma}\label{l5} The sequence $(x_{2m-2i})_{0{\le}i{\le}\frac{2m-k}2}$ is strictly increasing in $L_-$.%For every number $i\le m$ with $k\le 2i-2$ we get $x_{2i-2}>x_{2i}$ and hence $x_{2i-2}x_{2i}=x_{2i}\ne x_{2i-2}$ in $K$.
\end{lemma}

\begin{proof} Assume conversely that $x_{2i}\ge x_{2i-2}$ for some number $i\le m$ and assume that $i$ is the largest possible number with this property. The definition of the number $m$ guarantees that $i<m$. Consequently, $x_{2i-2}\le x_{2i}$ and $x_{2i}>x_{2i+2}$. Taking into account that $x_{2i-2},x_{2i},x_{2i+2}\in L_-$ and applying Lemma~\ref{l1}, we get  $x_{2i-2}x_{2i}=\min\{x_{2i-2},x_{2i}\}=x_{2i-2}$ and $x_{2i}x_{2i+2}=\min\{x_{2i},x_{2i+2}\}=x_{2i+2}$.

Now consider the elements $x_{2i-1},x_{2i+1}\in L_{+}$. If $x_{2i-1}\le x_{2i+1}$, then $x_{2i-1}x_{2i+1}=x_{2i+1}$ by Lemma~\ref{l2} and by Lemma~\ref{l4} the alternating word $x_{2i-2}x_{2i-1}x_{2i}x_{2i+1}=x_{2i-2}x_{2i+1}$ is reducible as $x_{2i-2}x_{2i}=x_{2i-2}$ and $x_{2i-1}x_{2i+1}=x_{2i+1}$.

If $x_{2i-1}>x_{2i+1}$, then $x_{2i-1}x_{2i+1}=x_{2i-1}$ by Lemma~\ref{l2}. Lemma~\ref{l4} guarantees that the alternating word $x_{2i-1}x_{2i}x_{2i+1}x_{2i+2}=x_{2i-1}x_{2i+2}$ is reducible as $x_{2i-1}x_{2i+1}=x_{2i-1}$ and $x_{2i}x_{2i+2}=x_{2i+2}$.

Therefore, the word $x_k\cdots x_n$ contains a reducible subword and hence is reducible, which contradicts the choice of this word. This contradiction shows that $x_{2i-2}<x_{2i}$ and hence $x_{2i-2}x_{2i}=\min\{x_{2i-2},x_{2i}\}=x_{2i-2}\ne x_{2i}$ according to Lemma~\ref{l1}.
\end{proof}

\begin{lemma}\label{l6} The sequence $(x_{2m-1-2i})_{0{\le}i{\le}\frac{2m{-}k{-}1}2}$ is strictly decreasing in $L_+$.%For every number $i<m$ with $k\le 2i-1$ we get $x_{2i-1}<x_{2i+1}$ and hence $x_{2i-1}x_{2i+1}=x_{2i+1}\ne x_{2i-1}$ in $K$.
\end{lemma}

\begin{proof}  Assume conversely that $x_{2i-1}\ge x_{2i+1}$ for some number $i<m$ with $k\le 2i-1$. Since $x_{2i-1},x_{2i+1}\in L_+$, Lemma~\ref{l2} implies that $x_{2i-1}x_{2i+1}=\max\{x_{2i-1},x_{2i+1}\}=x_{2i-1}$. By Lemma~\ref{l5}, $x_{2i}x_{2i+2}=x_{2i+2}$. Then by Lemma~\ref{l4} the alternating word $x_{2i-1}x_{2i}x_{2i+1}x_{2i+2}$ is equal to $x_{2i-1}x_{2i+2}$. This implies that the word $x_k\cdots x_n$ is reducible, which contradicts the choice of this word. This contradiction shows that $x_{2i-1}<x_{2i+1}$ and hence $x_{2i-1}x_{2i+1}=\max\{x_{2i-1},x_{2i+1}\}=x_{2i+1}\ne x_{2i-1}$ according to Lemma~\ref{l2}.
\end{proof}

Next, we consider the subword $x_{2m+1}\cdots x_n$ of the word $w=x_k\cdots x_n$.

\begin{lemma}\label{l7} The sequence $(x_{2m+2i})_{1{\le}i{\le}\frac{n{-}2m}2}$ is strictly increasing in $L_-$.%For every  number $i>m$ with $2i+2\le n$ we get $x_{2i}<x_{2i+2}$ and hence  $x_{2i}x_{2i+2}=\min\{x_{2i},x_{2i+2}\}=x_{2i}\ne x_{2i+2}$ in $K$.
\end{lemma}

\begin{proof} Assume conversely that $x_{2i}\ge x_{2i+2}$ for some number $i>m$ with $2i+2\le n$. We can assume that $i$ is the smallest possible number with this property. Then either $i=m+1$ or else $x_{2i-2}<x_{2i}$. If $i=m+1$, then $x_{2m}=x_{2i-2}\le x_{2i}$ by the choice of $m$. In both cases we get $x_{2i-2}\le x_{2i}$, which implies $x_{2i-2}x_{2i}=\min\{x_{2i-2},x_{2i}\}=x_{2i-2}$ according to Lemma~\ref{l1}. The same Lemma~\ref{l1}  implies that $x_{2i}x_{2i+2}=\min\{x_{2i},x_{2i+2}\}=x_{2i+2}$.

Now consider the elements $x_{2i-1}, x_{2i+1}\in L_+$. If $x_{2i-1}\ge x_{2i+1}$, then $x_{2i-1}x_{2i+1}=\max\{x_{2i-1},x_{2i+1}\}=x_{2i-1}$ and the alternating word $x_{2i-1}x_{2i}x_{2i+1}x_{2i+2}=x_{2i-1}x_{2i+2}$ is reducible according to Lemma~\ref{l4}. If $x_{2i-1}<x_{2i+1}$, then $x_{2i-1}x_{2i+1}=\max\{x_{2i-1},x_{2i+1}\}=x_{2i+1}$ and the alternating word $x_{2i-2}x_{2i-1}x_{2i}x_{2i+1}=x_{2i-2}x_{2i+1}$ is reducible by Lemma~\ref{l4}. But this contradicts the irreducibility of the word $x_k\cdots x_n$. So, $x_{2i}<x_{2i+2}$ and $x_{2i}x_{2i+2}=\min\{x_{2i}x_{2i+2}\}=x_{2i}\ne x_{2i+2}$ according to Lemma~\ref{l1}.
\end{proof}

\begin{lemma}\label{l8} The sequence $(x_{2m+1+2i})_{0{\le}i{\le}\frac{n{-}2m-1}2}$ is strictly decreasing in $L_+$.%For every $i>m$ with $2i+1\le n$ we get $x_{2i-1}>x_{2i+1}$ and hence  $x_{2i-1}x_{2i+1}=\max\{x_{2i-1},x_{2i+1}\}=x_{2i-1}\ne x_{2i+1}$ in $K$.
\end{lemma}

\begin{proof} Assume conversely that $x_{2i-1}\le x_{2i+1}$ for some $i>m$ with $2i+1\le n$. Then $x_{2i-1}x_{2i+1}=\max\{x_{2i-1},x_{2i+1}\}=x_{2i+1}$ according to Lemma~\ref{l2}. If $i=m+1$, then
$x_{2i-2}x_{2i}=x_{2m}x_{2i}=\min\{x_{2m},x_{2i}\}=x_{2m}=x_{2i-2}$ by the choice of the number $m$. If $i>m+1$, then $x_{2i-2}x_{2i}=x_{2i-2}$ by Lemma~\ref{l7}. By Lemma~\ref{l4} the alternating word $x_{2i-2}x_{2i-1}x_{2i}x_{2i+1}=x_{2i-2}x_{2i+1}$ is reducible, which contradicts the irreducibility of the word $x_k\cdots x_n$. This contradiction shows that $x_{2i-1}>x_{2i+1}$ and hence $x_{2i-1}x_{2i+1}=\max\{x_{2i-1},x_{2i+1}\}=x_{2i-1}\ne x_{2i+1}$ according to Lemma~\ref{l2}.
\end{proof}

Now we are ready to complete the proof of Theorem~\ref{represent}. Five cases are possible.
\smallskip

1) $2m=n$. In this case $x_k\dots x_n=x_k\dots x_{2m}$ is a $V_{\pm}$-word by Lemmas~\ref{l5} and \ref{l6}.
\smallskip

2) $2m+1=n$ and $k=2m$.  In this case $x_k\dots x_n=x_{2m}x_{2m+1}$ is a $V_\mp$-word.
\smallskip

3) $2m+1=n$ and $k\le 2m-1$. This case has three subcases.
\smallskip

3a) If $x_{2m-1}<x_{2m+1}$, then $x_k\dots x_n=x_k\dots x_{2m-1}x_{2m}x_{2m+1}$ is a $V_{\mp}$-word by Lemmas~\ref{l5} and \ref{l6}.

3b) If $x_{2m-1}>x_{2m+1}$, then $x_k\dots x_n=x_k\dots x_{2m-1}x_{2m}x_{2m+1}$ is a $V_{\pm}$-word by Lemmas~\ref{l5} and \ref{l6}.

3c) If $x_{2m-1}=x_{2m+1}$, then $x_k\dots x_n=x_k\dots x_{2m-1}x_{2m}x_{2m+1}$ is a $W_+$-word by Lemmas~\ref{l5} and \ref{l6}.
\smallskip

4) $2m+2\le n$ and $k=2m$. Since $x_{2m}\le x_{2m+2}$, this case has two subcases.
\smallskip

4a) If $x_{2m}<x_{2m+2}$, then $x_k\dots x_n=x_{2m}x_{2m+1}x_{2m+2}\dots x_n$ is a $V_{\mp}$-word by Lemmas~\ref{l7} and \ref{l8}.

4b) If $x_{2m}=x_{2m+2}$, then $x_k\dots x_n=x_{2m}x_{2m+1}x_{2m+2}\dots x_n$ is a $W_-$-word  by Lemmas~\ref{l7} and \ref{l8}.
\smallskip

5) $2m+2\le n$ and $k\le 2m-1$. This case has four subcases.
\smallskip

5a) $x_{2m}<x_{2m+2}$ and $x_{2m-1}<x_{2m+1}$. In this case
$x_k\dots x_n=x_k\dots x_{2m}x_{2m+1}\dots x_n$ is a $V_{\mp}$-word by Lemmas~\ref{l5}---\ref{l8}.

5b) $x_{2m}<x_{2m+2}$ and $x_{2m-1}>x_{2m+1}$. In this case
$x_k\dots x_n=x_k\dots x_{2m-1}x_{2m}\dots x_n$ is a $V_{\pm}$-word  by Lemmas~\ref{l5}---\ref{l8}.

5b) $x_{2m}<x_{2m+2}$ and $x_{2m-1}=x_{2m+1}$. In this case
$x_k\dots x_{2m-1}x_{2m}x_{2m-1}\dots x_n$ is a $W_+$-word  by Lemmas~\ref{l5}--\ref{l8}.

5c) $x_{2m}=x_{2m+2}$. In this case we shall prove that $x_{2m-1}>x_{2m+1}$.
Assuming  that $x_{2m-1}\le x_{2m+1}$ we can apply Lemma~\ref{l2} to conclude that $x_{2m-1}x_{2m+1}=\max\{x_{2m-1},x_{2m+1}\}=x_{2m-1}$. It follows from $x_{2m}=x_{2m+2}$ that $x_{2m}x_{2m+2}=x_{2m+2}$. By Lemma~\ref{l4}, the alternating word $x_{2m-1}x_{2m}x_{2m+1}x_{2m+2}$ is reducible, which is a contradiction. So, $x_{2m-1}>x_{2m+1}$ and hence
$x_k\dots x_n=x_k\dots x_{2m}x_{2m+1}x_{2m+2}\dots x_n$ is a $W_-$-word  by Lemmas~\ref{l5}--\ref{l8}.
\smallskip

Therefore each irreducible word in $FS_L$ is a Kuratowski word, which implies that $|K|\le |\K_L|$. If the set $L$ is finite, then the set $\K_L$ of Kuratowski words over $L$ has cardinality $|\K_L|=K(|L_-|,|L_+|)$, see Theorem~\ref{kuratwords}.
\end{proof}

\section{Separation of Kuratowski words by homomorphisms}\label{s:separ}

In the preceding section we proved that any element of a Kuratowski monoid $K$ with a linear generating set $L$ can be represented by a Kuratowski word $w\in\K_L$. In this section we shall prove that Kuratowski words can be separated by homomorphisms into the Kuratowski monoids of suitable $2$-topological spaces.

Given an $n$-topological space $\mathbf X=\big(X,(\tau_i)_{i\in n}\big)$, observe that the linear generating set $$L({\mathbf X})=\{\breve\tau_i\}_{i\in n}\cup\{1_X\}\cup\{\bar\tau_i\}_{i\in n}$$ of its Kuratowski monoid $\IK(\mathbf X)$ is symmetric.

This observation motivates the following definition. A {\em $*$-linearly ordered set} is a linearly ordered set $L$ endowed with an involutive bijection $*:L\to L$, $*:\ell\mapsto\ell^*$, that has a unique fixed point $1\in L$ and is {\em decreasing} in the sense that for any elements $x<y$ in $L$ we get $x^*>y^*$. Each $*$-linearly ordered set $L$ is pointed -- the unit of $L$ is the unique fixed point of the involution $*:L\to L$. Observe that the structure of a $*$-linearly ordered set $L$ is determined by the structure of its negative part $L_-$.

A map $f:L\to\Lambda$ between two $*$-linearly ordered sets $L,\Lambda$ will be called a {\em $*$-morphism} if $f$ is {\em monotone} (in the sense that for any elements $x\le y$ of $L$ we get $f(x)\le f(y)$) and {\em preserves the involution} (in the sense that $f(x^*)=f(x)^*$ for every $x\in L$. Since $f(1)=f(1^*)=f(1)^*$, the image $f(1)$ of the unit of $L$ coincides with the unit of $\Lambda$). Observe that each $*$-morphism $f:L\to\Lambda$ is uniquely determined by its restriction $f|L_-$.

For a $*$-linearly ordered set $L$, the involution $*:L\to L$ of $L$ has a unique extension to an involutive semigroup isomorphism $*:FS_L\to FS_L$ of the free semigroup over $L$. The image of a word $w\in FS_L$ under this involutive isomorphism will be denoted by $w^*$.

%For an $n$-topological space $\mathbf X=(X,(\tau_i)_{i\in n})$, the linear generating set $L({\mathbf X})=\{\breve\tau_i\}_{i\in n}\cup\{1_X\}\cup\{\bar\tau_i\}_{i\in n}$ of its Kuratowski monoid $\IK(\mathbf X)$ has a natural structure of a $*$-linearly ordered set endowed with the involution $*:L(\mathbf X)\to L(\mathbf X)$ mapping each interior operator $\breve\tau_i$ onto the closure operator $\bar\tau_i$ and vice versa.

Let $\mathbf X=(X,\Tau)$ be a polytopological space and $$L(\mathbf X)=\{\breve\tau:\tau\in\Tau\}\cup\{1_X\}\cup\{\bar\tau:\tau\in\Tau\}$$ be the linear generating set of the Kuratowski monoid $\IK(\mathbf X)$ of $\mathbf X$. Observe that each topology $\tau$ is determined by its interior operator $\breve\tau$ (since $\tau=\{\breve\tau(A):A\subset X\}$). This implies that the interior operators $\breve\tau$, $\tau\in\Tau$, are pairwise distinct. The same is true for the closure operators $\bar\tau$, $\tau\in\Tau$. This allows us to define a bijective involution $*:L(\mathbf X)\to L(\mathbf X)$ letting $\breve\tau^*=\bar\tau$ and $\bar\tau^*=\breve\tau$ for every $\tau\in\Tau$.
This involution turns $L(\mathbf X)$ into a $*$-linearly ordered set.

Let $L$ be a $*$-linearly ordered set. Choose any point $c\notin L$ and consider the free semigroup $FS_{L\cup\{c\}}$ over the set $L\cup\{c\}$. This semigroup consists of words in the alphabet $L\cup\{c\}$. Let $\mathbf X=(X,\Tau)$ be a polytopological space and $L(\mathbf X)$ be the linear generating set of the Kuratowski monoid $\IK(\mathbf X)$ of $\mathbf X$. Let $c_X:\mathcal P(X)\to\mathcal P(X)$, $c_X:A\mapsto X\setminus A$, denote the operator of taking complement.

Given any $*$-morphism $f:L\to L(\mathbf X)$ let $\hat f:FS_{L\cup\{c\}}\to \IK_2(\mathbf X)$ be a (unique) semigroup homomorphism such that $\hat f(1)=1_X$, $\hat f(c)=c_X$, and $\hat f(\ell)=f(\ell)$ for $\ell\in L$. The homomorphism $\hat f$ will be called the {\em Kuratowski extension} of $f$.

Observe that $\hat f(\K_L)\subset \IK(\mathbf X)$. In the semigroup $FS_{L\cup\{c\}}$ consider the subset $$\widetilde \K_L=\K_L\cup\{cw:w\in\K_L\}\subset FS_{L\cup\{c\}}$$whose elements will be called {\em full Kuratowski words}.

\begin{theorem}\label{separate} For any $*$-linearly ordered set $L$ and any two distinct words $u,v\in\widetilde \K_L$ there is a 2-topological space $\mathbf X$, and a $*$-morphism $f:L\to L(\mathbf X)$ whose Kuratowski extension $\hat f:FS_{L\cup\{c\}}\to\IK_2(\mathbf X)$ separates the words $u,v$ in the sense that $\hat f(u)\ne \hat f(v)$.
\end{theorem}

\begin{proof} In most of cases the underlying set of the $2$-topological space $\mathbf X$ will be a set $X=\{x,y\}$ containing two pairwise distinct points $x,y$ and the topologies of $\mathbf X$ are equal to one of four possible topologies on $X$:
\begin{itemize}
\item $\tau_d=\big\{\emptyset,\{x\},\{y\},X\big\}$, the discrete topology on $X$;
\item $\tau_a=\big\{\emptyset,X\big\}$, the anti-discrete topology on $X$;
\item $\tau_x=\{\emptyset,\{x\},X\}$;
\item $\tau_y=\{\emptyset,\{y\},X\}$.
\end{itemize}

Fix any two distinct words $u,v\in \widetilde \K_L$ and consider four cases.
\smallskip

1) $u\in\K_L$ and $v\notin\K_L$. In this case consider the 2-topological space $\mathbf X=\big(X,(\tau_a,\tau_d)\big)$. Then for the $*$-morphism $f:L\to \{1_X\}\subset L(\mathbf X)$ we get $\hat f(u)=1_X\ne c_X=\hat f(v)$.
\smallskip

2) $v\in\K_L$ and $u\notin\K_L$. In this case take the 2-topological space $\mathbf X$ from the preceding case and observe that for the $*$-morphism $f:L\to \{1_X\}\subset L(\mathbf X)$ we get $\hat f(u)=c_X\ne 1_X=\hat f(v)$.
\smallskip

3) $u,v\in\K_L$. Denote by $u_0,v_0\in L$ the last letters of the words $u,v$, respectively. Consider two cases.
\smallskip

3a) $u_0\ne v_0$. We lose no generality assuming that $u_0<v_0$ (in the linearly ordered set $L$).
\smallskip

Five subcases are possible:
\smallskip

3aa) $u_0=1$ and $v_0\in L_+$. In this case consider the 2-topological space $\mathbf X=\big(\{x,y\},(\tau_a,\tau_d)\big)$ and the $*$-morphism $f:L\to L(\mathbf X)$ assigning to each $\ell\in L_-$ the operator $\breve \tau_a$. Then for the subset $A=\{x\}$ of $\mathbf X$ and the operators $\hat u=\hat f(u)$ and $\hat v=\hat f(v)$, we get
$\hat u(A)=A\ne X=\bar\tau_a(A)=\hat v(A)$, which implies that $\hat f(u)\ne\hat f(v)$.
\smallskip

3ab) $u_0\in L_-$ and $v_0=1$. In this case, take the 2-topological space $\mathbf X$, the subset $A=\{x\}$, and the $*$-morphism $f:L\to  L(\mathbf X)$ from the preceding case. Then $\hat u(A)=\emptyset\ne A=\hat v(A)$, which implies that $\hat f(u)\ne\hat f(v)$.
\smallskip

3ac) $u_0\in L_-$ and $v_0\in L_+$. In this case, take the 2-topological space $\mathbf X$, the subset $A=\{x\}$ and the $*$-morphism $f:L\to L(\mathbf X)$ from case (3aa). Then $\hat u(A)=\emptyset\ne X=\hat v(A)$, which implies that $\hat f(u)\ne\hat f(v)$.
\smallskip

3ad) $u_0,v_0\in L_-$. Consider the 2-topological space $\mathbf X=\big(\{x,y\},(\tau_a,\tau_x)\big)$ and the $*$-morphism $f:L\to L(\mathbf X)$ assigning to each $\ell\in L_-$ the operator $$f(\ell)=\begin{cases}\breve \tau_a&\mbox{if $\ell\le u_0$};\\
\breve\tau_x&\mbox{if $\ell>u_0$}.
\end{cases}
$$ Put $\hat u=\hat f(u)$ and $\hat v=\hat f(v)$. Observe that for the subset $A=\{x\}$ we get
 $\hat u_0(A)=\breve\tau_a(A)=\emptyset$ and hence $\hat u(A)=\emptyset$.
Next, we evaluate $\hat v(A)$. Write the Kuratowski word $v$ as $v=v_q\dots v_0$ where $v_0,\dots,v_q\in L\setminus\{1\}$.  If $q=0$, then $\hat v(A)=\hat v_0(A)=\breve\tau_x(A)=A\ne \emptyset=\hat u(A)$. If $q>0$, then $\hat v_1\hat v_0(A)=\hat v_1(A)\in \{\bar\tau_a(A),\bar\tau_x(A)\}=\{X\}$ and hence $\hat v(A)=X\ne\emptyset=\hat u(A)$. This yields the desired inequality $\hat f(u)\ne\hat f(v)$.
\smallskip

3ae) $u_0,v_0\in L_+$. In this case we can consider the conjugated words $u^*$ and $v^*$ and observe that their last letters are distinct and belong to the set $L_-$. By the preceding item, there are a $2$-topological space $\mathbf X$ and a $*$-morphism $f:L\to L(\mathbf X)$ such that $\hat f(u^*)\ne\hat f(v^*)$. Then $\hat f(u)^*=\hat f(u^*)\ne\hat f(v^*)=\hat f(v)^*$ and hence $\hat f(u)\ne\hat f(v)$.
\smallskip

Next, consider the case:

3b) $u_0=v_0$. It follows that $u_0=v_0\in L\setminus\{1\}$. Write the Kuratowski words $u,v$ as $u=u_p\dots u_0$ and $v=v_q\dots v_0$ where $u_i,v_i\in L\setminus\{1\}$. For the sake of consistency it will be convenient to assume that
$u_i=1$ for every integer $i\notin \{0,\dots,p\}$ and $v_j=1$ for avery integer $j\notin\{0,\dots,q\}$. Un particular, $u_{-1}=v_{-1}=1$.

Since $u\ne v$, the number $k=\min\{i\ge0:u_i\ne v_i\}$ is well-defined and does not exceed $\max\{p,q\}$. It follows from $u_0=v_0$ that $k\ge 1$. By the definition of $k$ we also get the equality $u_{k-1}\dots u_0=v_{k-1}\dots v_0$. First we consider the case $u_{k-1}=v_{k-1}\in L_+$.
Since $u_k\ne v_k$, we lose no generality assuming that $u_k<v_k$. Since $u,v$ are alternating words, the inclusion $u_{k-1}=v_{k-1}\in L_+$ implies $u_k,v_k\in L_-\cup\{1\}$ and $u_k\in L_-$.
\smallskip

Two cases are possible.
\smallskip

3ba) $u_k<u_{k-2}$ (this case includes also the case of $k=1$ in which $u_1<1=u_{-1}$). Since $u\in \K_L$, the inequality
$u_k<u_{k-2}$ implies that the sequence $(u_{k-2i})_{0{\le} i{\le} \frac{k}2}$ is strictly increasing in $L_-$ and the sequence
$(u_{k-1-2i})_{0{\le} i{\le} \frac{k-1}2}=(v_{k-1-2i})_{0{\le} i{\le} \frac{k-1}2}$ is a strictly decreasing  in $L_+$.
\smallskip

Now consider two subcases.
\smallskip

3baa) $v_{k+1}>v_{k-1}$. In this case $(v^*_{k+1-2i})_{0\le i\le\frac{k+1}2}$ is a strictly increasing sequence in $L_{-}$. Depending on the relation between the elements $u_k$ and $v^*_{k+1}$ of $L_-$ we shall distinguish two subsubcases.
\smallskip

3baaa) $v^*_{k+1}\le u_k$. In this case consider the 2-topological space $\mathbf X=\big(X,(\tau_a,\tau_y)\big)$ where $X=\{x,y\}$ and define the $*$-morphism $f:L\to L(\mathbf X)$ assigning to each $\ell\in L_-$ the operator
$$f(\ell)=\begin{cases}
\breve\tau_a&\mbox{if $\ell\le v^*_{k+1}$;}\\
\breve\tau_y&\mbox{if $v^*_{k+1}<\ell\le u_k$;}\\
1_X&\mbox{if $u_{k}<\ell$}.
\end{cases}
$$Consider the subset $A=\{x\}\subset X$.
Since the sequence $(u_{k-2i})_{0\le i\le\frac k2}$ is strictly increasing in $L_-$, for every positive number $i\le\frac{k}2$ we get $v_{k-2i}=u_{k-2i}>u_k$ and hence
$\hat u_{k-2i}=\hat v_{k-2i}=1_X$, which implies that $\hat u_{k-2i}(A)=\hat v_{k-2i}(A)=A$. On the other hand, for every positive $i\le\frac{k+1}2$ we get $u^*_{k+1-2i}=v^*_{k+1-2i}>v^*_{k+1}$ and hence
$\hat u_{k+1-2i}=\hat v_{k+1-2i}\in\{\bar\tau_y,1_X\}$, which implies $\hat u_{k+1-2i}(A)=\hat v_{k+1-2i}(A)\in\{\bar\tau_y(A),1_X(A)\}=\{A\}$. So, $\hat u_{k-1}\cdots\hat u_0(A)=\hat v_{k-1}\cdots\hat v_0(A)=A$.

Observe that $\hat u_k\cdots\hat u_1(A)=\hat u_k(A)=\breve \tau_y(A)=\emptyset$ and hence $\hat u(A)=\emptyset$. On the other hand, the inequality $v_k>u_k$ implies that $\hat v_k=1_X$ and then $\hat v_k\cdots \hat v_0(A)=\hat v_k(A)=A$ and $\hat v_{k+1}\cdots\hat v_0(A)=\hat v_{k+1}(A)=\bar\tau_a(A)=X$ and then $\hat v(A)=X\ne\emptyset=\hat u(A)$, which implies that $\hat f(u)\ne\hat f(v)$.
\smallskip

3baab) $v^*_{k+1}>u_k$. In this case consider the 2-topological space $\mathbf X=\big(X,(\tau_a,\tau_x)\big)$ where $X=\{x,y\}$ and define a $*$-morphism $f:L\to L(\mathbf X)$ assigning to each $\ell\in L_-$ the operator
$$f(\ell)=\begin{cases}
\breve\tau_a&\mbox{if $\ell\le u_k$;}\\
\breve\tau_x&\mbox{if $u_k<\ell\le v^*_{k+1}$;}\\
1_X&\mbox{if $v^*_{k+1}<\ell$}.
\end{cases}
$$Consider the subset $A=\{x\}\subset X$.
Since the sequence $(u_{k-2i})_{0\le i\le \frac{k}2}$ is strictly increasing in $L_-$,  for every positive $i\le\frac k2$ we get $v_{k-2i}=u_{k-2i}>u_k$ and hence
$\hat u_{k-2i}=\hat v_{k-2i}\in\{\breve\tau_x,1_X\}$, which implies that $\hat u_{k-2i}(A)=\hat v_{k-2i}(A)\in\{\breve\tau_x(A),1_X(A)\}=\{A\}$.
Since the sequence $(v^*_{k+1-2i})_{0\le i\le \frac{k+1}2}$ is strictly increasing in $L_-$,
 for every positive $i\le \frac{k+1}2$ we get $u^*_{k+1-2i}= v^*_{k+1-2i}>v^*_{k+1}$ and hence
$\hat u_{k+1-2i}=\hat v_{k+1-2i}=1_X$, which implies $\hat u_{k+1-2i}(A)=\hat v_{k+1-2i}(A)=A$. So, $\hat u_{k-1}\cdots\hat u_0(A)=\hat v_{k-1}\cdots\hat v_0(A)=A$.

Observe that $\hat u_k\cdots\hat u_1(A)=\hat u_k(A)=\breve \tau_a(A)=\emptyset$ and hence $\hat u(A)=\emptyset$. Next, we evaluate $\hat v(A)$. If $v_k=1$, then $\hat v(A)=\hat v_{k-1}\cdots\hat v_0(A)=A\ne\emptyset =\hat u(A)$. If $v_k\ne 1$ but $v_{k+1}=1$, then $u_k<v_k$ implies that $\hat v_k\in\{\breve\tau_x,1_X\}$ and hence
 $\hat v(A)=\hat v_k\cdots \hat v_0(A)=\hat v_k(A)\in\{\breve\tau_x(A),1_X(A)\}=\{A\}$ and
 hence $\hat v(A)=A\ne\emptyset=\hat u(A)$. If $v_{k+1}\ne 1$, then
 $\hat v_{k+1}\cdots\hat v_1(A)=\hat v_{k+1}(A)=\bar\tau_x(A)=X$ and then $\hat v(A)=X\ne\emptyset=\hat u(A)$. This yields the desired inequality $\hat f(u)\ne\hat f(v)$.
\smallskip

Next, consider the subcase:
\smallskip

3bab) $v_{k+1}=v_{k-1}$. In this case $v\in \W_+$ and hence the sequences $(v_{k-2i})_{0{\le}i{\le}\frac k2}$ and $(v_{k+2i})_{0{\le}i{\le}\frac{q-k}2}$ are strictly increasing in $L_-$ whereas the sequences $(v_{k-1-2i})_{0{\le}i{\le}\frac{k-1}2}$ and $(v_{k+1+2i})_{0{\le}i{\le}\frac{q-k-1}2}$ are strictly decreasing in $L_+$. Consider the 2-topological space $\mathbf X=\big(\{x,y\},(\tau_a,\tau_y)\big)$ and define a $*$-morphism $f:L\to L(\mathbf X)$ assigning to each element $\ell\in L_-$ the operator
$$f(\ell)=\begin{cases}
\breve\tau_y&\mbox{if $\ell\le u_k$;}\\
1_X&\mbox{if $\ell>u_k$}.
\end{cases}
$$
Also consider the subset $A=\{x\}$ in the 2-topological space $\mathbf X$.

Taking into account that $u_k<v_k\le v_{k+2i}$ for any $i\in\IZ$ with $0\le k+2i\le q$, we conclude that $\hat v_{k+2i}=1_X$. On the other hand, for every $i\in\IZ$ with $1\le k+1+2i\le q$ we get $\hat v_{k+1+2i}\in\{\bar\tau_y,1_X\}$ and hence $\hat v_{k+1+2i}(A)\in\{\bar\tau_y(A),\bar\tau_d(A)\}=\{A\}$. This implies that $\hat v(A)=A$.

On the other hand, $\hat u_k\cdots\hat u_0(A)=\hat u_k\hat v_{k-1}\cdots\hat v_0(A)=\hat u_k(A)=\breve\tau_y(A)=\emptyset$ and then $\hat u(A)=\emptyset\ne A=\hat v(A)$. This implies that $\hat f(u)\ne\hat f(v)$.
\smallskip

Finally, consider the subcase:
\smallskip

3bac) $v_{k+1}<v_{k-1}$. Taking into account that  $v\in \K_L$ and the sequence $(v_{k-1-2i})_{0\le i\le \frac{k-1}2}$ is strictly decreasing in $L_+$, we conclude that $v_{k-1}>v_{k-1+2i}$ for any non-zero integer number $i$ with $0\le k-1+2i\le q$ and $u_k<\min\{u_{k-2},v_k\}=\min\{v_{k-2},v_k\}\le v_{k+2i}$ for any integer number $i$ with $k+2i\in\{0,\dots,q\}$. Take the 2-topological space $\mathbf X$, the subset $A=\{x\}$, and the $*$-morphism $f:L\to L(\mathbf X)$ from the case (3bab).
By analogy with the preceding case it can be shown that $\hat v(A)=A\ne \emptyset=\hat u(A)$ and hence $\hat f(u)\ne\hat f(v)$.
This completes the proof of case (3ba).
\smallskip

So, now we consider the case:
\smallskip

3bb) $u_k\ge u_{k-2}$. This case is more difficult and requires to consider a set $X=\{x,y,z\}$ of cardinality $|X|=3$ endowed with the topologies:
\begin{itemize}
%\item $\tau_d$, the discrete topology on $X$;
%\item $\tau_a=\{\emptyset,X\}$, the anti-discrete topology on $X$;
\item $\tau_{x,z}=\big\{\emptyset,\{x\},\{z\},\{x,z\},X\big\}$,
\item $\tau_{x,z,yz}=\big\{\emptyset,\{x\},\{z\},\{y,z\},\{x,z\},X\big\}$,
\item $\tau_{x,z,xy}=\big\{\emptyset,\{x\},\{z\},\{x,y\},\{x,z\},X\}$.
\end{itemize}
In $X$ consider the subset $A=\{x\}$.

The inequality $u_{k-2}\le u_k\in L_-$ and the inclusion $u\in\K_L$ imply that
the sequence $\big(u_{k-1+2i}\big)_{0\le i\le \frac{p-k+1}2}$ is strictly decreasing in $L_+$. By analogy, the (strict) inequality $v_{k-2}=u_{k-2}\le u_k<v_k$ implies that sequence $\big(v_{k-2+2i}\big)_{0\le i\le \frac{q-k+2}2}$ is strictly increasing in $L_-$ and $\big(v_{k-1+2i}\big)_{0\le i\le\frac{q-k+1}2}$ in strictly decreasing in $L_+$. Let $u^*_{k-1}$ be the letter in $L_-$, symmetric to the letter $u_{k-1}$ with respect to 1.
Depending on the relation between $u^*_{k-1}$ and $u_k$ we consider two subcases:
\smallskip

3bba) $u^*_{k-1}\le u_{k}$. In this case consider the 2-topological space $\mathbf X=\big(X,(\tau_{x,z},\tau_{x,z,yz})\big)$ where $X=\{x,y,z\}$ and define the $*$-morphism $f:L\to L(\mathbf X)$ assigning to every element $\ell\in L_-$ the operator
$$f(\ell)=\hat\ell:=\begin{cases}
\breve\tau_{x,z}&\mbox{if $\ell\le u^*_{k-1}$};\\
\breve\tau_{x,z,yz}&\mbox{if $u^*_{k-1}<\ell\le u_k$};\\
1_X&\mbox{if $u_k<\ell$}.
\end{cases}
$$Consider the subset $A=\{x\}$ of $X$.
Observe that for every $\ell\in L$ we get $\{x\}\subset \hat\ell(\{x\})\subset \hat\ell(\{x,y\})\subset\{x,y\}$. Then $$\{x\}\subset \hat u_{k-2}\hat u_{k-3}\cdots\hat u_0(A)\subset \{x,y\}$$ and $$\{x,y\}=\bar\tau_{x,z}(\{x\})=\hat u_{k-1}(\{x\})\subset \hat u_{k-1}\hat u_{k-2}\cdots \hat u_0(A)\subset\hat u_{k-1}(\{x,y\})=\bar\tau_{x,z}(\{x,y\})=\{x,y\}.$$ So, $\hat v_{k-1}\cdots\hat v_0(A)=\hat u_{k-1}\cdots \hat u_0(A)=\{x,y\}$ and
$\hat u_k\cdots\hat u_1(A)=\hat u_k(\{x,y\})=\breve \tau_{x,z,yz}(\{x,y\})=\{x\}=A$.
%It follows that $\hat u_{k-1}\cdots\hat u_1(A)=\hat u_{k-1}(\{x\})\in\{\bar\tau_{x,z,yz}(\{x\}),\bar\tau_d(\{x\})\}=\{A\}$ and $\hat u_k\cdots\hat u_1(A)=\hat u_k(A)=\breve\tau_{x,z,yz}(A)=A$.

Since the sequence $(u_{k-1+2i})_{0\le i\le\frac{p-k+1}2}$ is strictly decreasing in $L_+$, the dual sequence $(u^*_{k-1+2i})_{0\le i\le\frac{p-k+1}2}$ is strictly increasing in $L_-$. Then $\{\hat u_{k-1+2i}\}_{1\le i\le\frac{p-k-1}2}\subset\{\bar\tau_{x,z,yz},1_X\}$ and then $\hat u(A)=\hat u_p\cdots\hat u_{k+1}(A)\subset A$.

On the other hand, $v_k>u_k$ implies that $\hat v_k\cdots\hat v_1(A)=\hat v_k(\{x,y\})=1_X(\{x,y\})=\{x,y\}$. Taking into account that $u_k<v_k$ and the sequence $(v_{k+2i})_{0\le i\le \frac{q-k}2}$ is strictly increasing in $L_-$, we conclude that $\{v_{k+2i}\}_{0\le i\le \frac{q-k}2}\subset\{1_X\}$, which implies that $\hat v(A)=\hat v_k\cdots\hat v_0(A)\supset\{x,y\}\ne \{x\}\supset\hat u(A)$. This yields the desired inequality $\hat f(u)\ne\hat f(v)$.
\smallskip

3bbb) $u_k<u^*_{k-1}$. In this case consider the 2-topological space $\mathbf X=\big(X,(\tau_{x,z},\tau_{x,z,xy})\big)$ where $X=\{x,y,z\}$. Define a $*$-morphism $f:L\to L(\mathbf X)$ assigning to every element $\ell\in L_-$ the operator
$$f(\ell)=\begin{cases}
\breve\tau_{x,z}&\mbox{if $\ell\le u_k$};\\
\breve\tau_{x,z,xy}&\mbox{if $u_k<\ell\le u^*_{k-1}$};\\
1_X&\mbox{if $u^*_{k-1}<\ell$}.
\end{cases}
$$Consider the subset $A=\{x\}$ of $X$. It follows that $$\hat u_{k-1}\cdots\hat u_0(A)\in\{\hat u_{k-1}(\{x\}),\hat u_{k-1}(\{x,y\})\}=\{\bar\tau_{x,z,xy}(\{x\}),\bar\tau_{x,z,xy}(\{x,y\})\}=\big\{\{x,y\}\big\}$$ and hence $\hat u_k\cdots\hat u_0(A)=\hat u_k(\{x,y\})=\breve\tau_{x,z}(\{x,y\})=\{x\}$.
Taking into account that the sequence\break $(u^*_{k-1+2i})_{0\le i\le \frac{p{-}k{+}1}2}$ is strictly increasing in $L_-$, we conclude that $\hat u_{k-1+2i}=1_X$ for all positive $i\le \frac{p-k+1}2$. This implies that $\hat u(A)=\hat u_p\cdots\hat u_{k+1}(A)\subset A=\{x\}$.

On the other hand, $\hat v_k\cdots\hat v_0(A)=\hat v_k(\{x,y\})\in \{\breve\tau_{x,z,xy}(\{x,y\}),1_X(\{x,y\})\}=\big\{\{x,y\}\big\}$. Taking into account that $u_k<v_k$ and the sequence $(v_{k+2i})_{0{\le}i{\le}\frac{q-k}2}$ is strictly increasing, we conclude that $\{\hat v_{k+2i}\}_{0\le i\le \frac{q-k}2}\subset \{\breve\tau_{x,z,xy},1_X\}$, which implies that $\hat v(A)=\hat v_q\cdots\hat v_{k+1}(\{x,y\})\supset \{x,y\}$. So, $\hat u(A)\ne\hat v(A)$, which implies that $\hat f(u)\ne\hat f(v)$.
\smallskip

This completes the proof of case (3) under the assumption $u_{k-1}=v_{k-1}\in L_+$. If $u_{k-1}=v_{k-1}\in L_-$, then we can consider the dual words $u^*=u_p^*\cdots u_0^*$ and $v^*=v^*_q\cdots v^*_0$. For these dual words we get $u_i^*=v_i^*$ for $i<k$ and  $v_{k-1}^*=u_{k-1}^*\in L_+$. In this case the preceding proof yields an $2$-topological space $\mathbf X$ and a $*$-morphism $f:L\to L(\mathbf X)$ such that $\hat f(u)^*=\hat f(u^*)\ne\hat f(v^*)=\hat f(v)^*$, which implies that $\hat f(u)\ne\hat f(v)$.
\smallskip

Finally, consider the case:

4) $u,v\in\widetilde\K_L\setminus\K_L$. Then $u=cu'$ and $v=cv'$ for some distinct words $u',v'\in\K_L$. By the case (3), we can find a $2$-topological space $\mathbf X$ and a $*$-morphism $f:L\to L(\mathbf X)$ such that $\hat f(u')\ne\hat f(v')$. This means that $\hat f(u')(A)\ne\hat f(v')(A)$ for some subset $A\subset X$. Then $$\hat f(u)(A)=\hat f(cu')(A)=\hat f(c)\hat f(u')(A)=c_X(\hat f(u')(A))\ne c_X(\hat f(v')(A))=\hat f(cv')(A)=\hat f(v)(A),$$ which means that $\hat f(u)\ne\hat f(v)$.
\end{proof}

The following corollary of Theorem~\ref{separate} implies Theorem~\ref{exact} and shows that the upper bound in Theorem~\ref{main} is exact.

\begin{corollary}\label{separateL} For any $*$-linearly ordered set $L$ there is an $L_-$-topological space $\mathbf X=\big(X,(\tau_\ell)_{\ell\in L_-}\big)$ such that the unique semigroup homomorphism $\pi:FS_{L\cup\{c\}}\to\IK_2(\mathbf X)$ such that $\pi(1)=1_X$, $\pi(c)=c_X$, $\pi(\ell)=\breve\tau_\ell$, $\pi(\ell^*)=\bar\tau_\ell$ for $\ell\in L_-$ maps bijectively the set $\K_L$ onto $\IK(\mathbf X)$ and the set $\widetilde\K_L$ onto $\IK_2(\mathbf X)$. If the set $L_-$ has finite cardinality $n$, then the Kuratowski monoid $\IK(\mathbf X)$ of $\mathbf X$ has cardinality $|\IK(\mathbf X)|=K(n,n)=K(n)$ and the full Kuratowski monoid $\IK_2(\mathbf X)$ of $\mathbf X$ has cardinality $|\IK_2(\mathbf X)|=2\cdot K(n,n)=2K(n)$.
\end{corollary}

\begin{proof} Let $\widetilde\K_L^2\setminus\Delta=\{(u,v)\in\widetilde\K_L\times\widetilde\K_L:u\ne v\}$. By Theorem~\ref{separate}, for any distinct words $u,v\in\widetilde K_L$ there exist a $2$-topological space $\mathbf X_{u,v}=(X_{u,v},(\tau'_{u,v},\tau''_{u,v}))$ and a $*$-morphism $f_{u,v}:L\to L(\mathbf X)$ whose Kuratowski extension $\hat f_{u,v}:FS_{L\cup\{c\}}\to\IK_2(\mathbf X_{u,v})$ separates the words $u,v$ in the sense that $\hat f(u)\ne \hat f(v)$. This means that  $\hat f_{u,v}(u)(A_{u,v})\ne \hat f_{u,v}(v)(A_{u,v})$ for some subset $A_{u,v}\subset X_{u,v}$.
Let $\delta_{u,v}$ denote the discrete topology on the set $X_{u,v}$.

Define the $L_-$-topology $\tau_{u,v}:L_-\to\Top(X_{u,v})$ on $X_{u,v}$ by the formula
$$\tau_{u,v}(\ell)=\begin{cases}
\tau'_{u,v}&\mbox{ if $f(\ell)=\breve\tau'_{u,v}$};\\
\tau''_{u,v}&\mbox{ if $f(\ell)=\breve\tau''_{u,v}$};\\
\delta_{u,v} &\mbox{otherwise}.
\end{cases}
$$
Consider the $L_-$-topological space $\mathbf X^-_{u,v}=(X_{u,v},\tau_{u,v})$ and observe that its full Kuratowski monoid coincides with the full Kuratowski monoid of the 2-topological space $\mathbf X_{u,v}$. Moreover, the Kuratowski extension $\hat f_{u,v}:FS_{L\cup\{c\}}\to\IK_2(\mathbf X_{u,v})=\IK_2(\mathbf X^-_{u,v})$ of the morphism $f_{u,v}$ has the properties $\hat f_{u,v}(1)=1_{X_{u,v}}$, $\hat f_{u,v}(c)=c_{X_{u,v}}$, $\hat f_{u,v}(\ell)=\breve\tau_{u,v}(\ell)$, and $\hat f(\ell^*)=\bar\tau_{u,v}(\ell)$ for $\ell\in L_-$.

We lose no generality assuming that for any distinct pairs $(u,v),(u',v')\in\widetilde\K_L^2\setminus\Delta$ the sets $X_{u,v}$ and $X_{u',v'}$ are disjoint.
This allows us to consider the disjoint union
$$X=\bigcup\big\{X_{u,v}:(u,v)\in\widetilde \K_L^2\setminus\Delta\big\}$$
and the subset $$A=\bigcup \big\{A_{u,v}:(u,v)\in\widetilde \K_L^2\setminus\Delta\big\}$$ in $X$.
For every $\ell\in L_-$  consider the topology $\tau_\ell$ on the set $X$ generated by the base $\bigcup\{\tau_{u,v}(\ell):(u,v)\in\widetilde \K_L^2\setminus\Delta\}$. We claim that the $L_-$-topological space $\mathbf X=(X,\tau)$ (which is the direct sum of $L_-$-topological spaces $\mathbf X_{u,v}^-$) and the subset $A\subset X$ have the desired property: for any two distinct words $u,v\in\widetilde\K_L$ we get $\hat u(A)\ne \hat v(A)$, where $\hat u$ and $\hat v$ are the images of $u$ and $v$ under the (unique) semigroup homomorphism $\pi:FS_{L\cup\{c\}})\to\IK_2(\mathbf X)$ such that $\pi(1)=1_X$, $\pi(c)=c_X$, $\pi(\ell)=\breve\tau_\ell$, $\pi(\ell^*)=\bar\tau_\ell$ for $\ell\in L_-$. This follows from the fact that $\hat u(A)\cap X_{u,v}=\hat f_{u,v}(u)(A_{u,v})\ne\hat f_{u,v}(v)(A_{u,v})=\hat v(A)\cap X_{u,v}$.

This means that the restriction $\pi:\widetilde\K_L\to\IK_2(\mathbf X)$ is injective. By Theorem~\ref{represent}, $\pi(\K_L)=\IK(\mathbf X)$. Since $\IK_2(\mathbf X)=\IK(\mathbf X)\cup\{c_X\circ w:w\in\IK(\mathbf X)\}$, we get also that $\pi(\tilde\K_L)=\IK_2(\mathbf X)$. This means that the homomorphism $\pi$ maps bijectively the set $\K_L$ onto $\IK(\mathbf X)$ and the set $\widetilde\K_L$ onto $\IK_2(\mathbf X)$.

If the set $L_-$ has finite cardinality $n$, then the set $\K_L$ has cardinality $|\K_L|=K(n,n)=K(n)$ (according to Theorem~\ref{kuratwords}) and hence $|\IK(\mathbf X)|=|\K_L|=K(n,n)$ and $|\IK_2(\mathbf X)|=|\widetilde\K_L|=2\cdot K(n,n)=2\cdot K(n)$.
\end{proof}

\section{Free Kuratowski monoids}\label{s:free}

In this section we shall discuss free Kuratowski monoids over pointed linearly ordered sets. By a {\em pointed linearly ordered set} we understand a linearly ordered set $(L,\le)$ with a distinguished point $1\in L$ called the {\em unit} of $L$. The subsets $L_-=\{x\in L:x<1\}$ and $L_+=\{x\in L:1<x\}$ are called the {\em negative} and {\em positive} parts of $L$, respectively. A function $f:L\to \Lambda$ between two pointed linearly ordered sets is called a {\em morphism} if $f(1)=1$ and $f$ is monotone in the sense that $f(x)\le f(y)$ for any elements $x\le y$ of $L$.

Each pointed linearly ordered set $L$ determines the {\em free Kuratowski monoid $FK_L$} defined as follows. On the free semigroup $FS_L=\bigcup_{n=1}^\infty L^n$ consider the smallest compatible partial preorder $\preceq$ extending the linear order $\le$ of the set $L=L^1\subset FS_L$ and containing the pairs $(x,x1)$, $(x,1x)$, $(x1,x)$, $(1x,x)$, $(x,x^2)$, and $(x^2,x)$ for $x\in L$. The compatible partial preorder $\preceq$ generates the congruence $\rho_\preceq=\{(v,w)\in FS_L:v\preceq w,\;w\preceq v\}$ on $FS_L$ identifying the words $x,x1,1x,x^2$ for any $x\in L$. The quotient semigroup $FK_L=FS_L/\rho_\preceq$ endowed with the quotient partial order is called the {\em free Kuratowski monoid} generated by the pointed linearly ordered set $L$.
By $q_L:FS_L\to FK_L$ we shall denote the (monotone) quotient homomorphism. The restriction $\eta_L=q_L|L:L\to FK_L$ is called the canonical embedding of the pointed linearly ordered set $L$ into its free Kuratowski monoid. In Proposition~\ref{quot} we shall see that $\eta$ is indeed injective.

First we show that the free Kuratowski monoid $FK_L$ is free in the categorial sense.

\begin{proposition}\label{catfree} For any pointed linearly ordered set $L$ and any Kuratowski monoid $K$ with a linear generating set $\Lambda$ any morphism $f:L\to\Lambda$ determines a unique monotone semigroup homomorphism $\bar f:FK_L\to K$ such that $\bar f\circ\eta=f$.
\end{proposition}

\begin{proof} Let $\tilde f:FS_L\to K$ be the unique semigroup homomorphism extending the morphism $f:L\to\Lambda\subset K$. The partial order $\le$ of the Kuratowski monoid $K$ induces the compatible partial preorder $\leqslant$ on $FS_L$ defined by $u\leqslant v$ iff $\tilde f(u)\le\tilde f(v)$. The monotonicity of $f$ implies that the partial preorder $\leqslant$ contains the linear order of the set $L$. Then the minimality of the partial preorder $\preceq$ implies that $\preceq\subset \leqslant$. This allows us to find a unique monotone semigroup homomorphism $\bar f:FK_L\to K$ such that $\bar f\circ q_L=\tilde f$ and hence $\bar f\circ \eta_L=\bar f\circ q_L|L=\tilde f|L=f$.
\end{proof}

\begin{proposition}\label{quot} For any pointed linearly ordered set $L$ the quotient homomorphism $q_L:FS_L\to FK_L$ maps bijectively the set $\K_L$ of Kuratowski words onto $FK_L$.
\end{proposition}

\begin{proof} Theorem~\ref{represent} implies that $q_L(\K_L)=FK_L$. To show that $q_L|\K_L$ is injective, we shall apply Theorem~\ref{separateL}. Choose any injective morphism $e:L\to L^*$ of the pointed linearly ordered set $L$ into a $*$-linearly ordered set $L^*$.
Let $\tilde e:FS_L\to FS_{L^*}$ be the unique semigroup homomorphism extending the map $e$.
The injectivity of $e$ implies the injectivity of the homomorphism $\tilde e$.

By Proposition~\ref{catfree}, the morphism $e$ determines a unique monotone semigroup homomorphism $\bar e:FK_L\to FK_{L^*}$ such that $\bar e\circ\eta_L=\eta_{L^*}\circ e$. By Theorem~\ref{separateL}, there exists a polytopological space $\mathbf X$ and a $*$-morphism $f:L^*\to L(\mathbf X)$ whose Kuratowski extension $\hat f:FS_{L^*}\to \IK(\mathbf X)$ maps bijectively the set $\K_{L^*}$ onto $\IK(\mathbf X)$. By Proposition~\ref{catfree}, the $*$-morphism $f:L^*\to L(\mathbf X)\subset \IK(\mathbf X)$ determines a  (unique) monotone semigroup homomorphism $\bar f:FS_{L^*}\to\IK(\mathbf X)$ such that $\bar f\circ\eta_{L^*}=f$.
Thus we obtain the commutative diagram
$$\xymatrix{
L\ar[r]^e\ar[d]&L^*\ar[d]\\
\K_L\ar[r]^{\tilde e|\K_L}\ar[d]&\K_{L^*}\ar[d]\\
FS_L\ar[r]^{\tilde e}\ar[d]_{q_L}&FS_{L^*}\ar[d]^{q_{L^*}}\ar[rd]^{\hat f}\\
FK_L\ar[r]_{\bar e}&FK_{L^*}\ar[r]_{\bar f}&\IK(\mathbf X)
}
$$
in which the map $\hat f\circ\tilde e|\K_L$ is injective. Since $\hat f\circ\tilde e=\bar f\circ\bar e\circ q_L$, the injectivity of the map $\hat f\circ\tilde e|\K_L$ implies the injectivity of the map $q_L|\K_L$. Since $q_L(\K_L)=FK_L$, the restriction $q_L|\K_L:\K_L\to FK_L$ is bijective.
\end{proof}

Now we prove that the congruence $\rho_\preceq$ on $FS_L$ determining the free Kuratowski monoid can be equivalently defined in a more algebraic fashion.

\begin{proposition} For any pointed linearly ordered set $L$ the congruence $\rho_{\preceq}$ on the free semigroup $FS_L$ coincides with the smallest congruence $\rho$ on $FS_L$ containing the
pairs $(x,x1)$, $(x,1x)$, and $(x,x^2)$  for every $x\in L$ and the pairs $(x_1y_1x_2y_2,x_1y_2)$, $(y_2x_2y_1x_1,y_2x_1)$ for any points $x_1,x_2,y_1,y_2\in L$ with $x_1\le x_2\le 1\le y_1\le y_2$.
\end{proposition}

\begin{proof} First we prove that $\rho\subset\rho_\preceq$. The definition of the partial preorder
$\preceq$ implies that $\{(x,x1),(x,1x),(x,x^2):x\in L\}\subset\rho_\preceq$. Repeating the proofs of Lemma~\ref{l4},  we can show that $$\{(x_1y_1x_2y_2,x_1y_2),(y_2x_2y_1x_1,y_2x_1):x_1,x_2,y_2,y_2\in L,\;x_1\le x_2\le 1\le y_1\le y_2\}\subset \rho_\preceq.$$
Now the minimality of the congruence $\rho$ implies that $\rho\subset\rho_\preceq$.

Denote by $\rho^\natural:FS_L\to FS_L/\rho$ and $q_L:FS_L\to FS_L/\rho_\preceq=FK_L$ the quotient homomorphisms. Since $\rho\subset\rho_\preceq$, there is a unique homomorphism $h:FS_L/\rho\to FK_L$ making the following diagram commutative:
$$\xymatrix{
&\K_L\ar[d]^<<<<i\\
&FS_L\ar[ld]_{\rho^\natural}\ar[rd]^{q_L}\\
FS_L/\rho\ar[rr]_h&&FK_L}
$$
In this diagram by $i:\K_L\to FS_L$ we denote the identity inclusion of the set $\K_L$ of Kuratowski words into the free semigroup $FS_L$. The proof of Theorem~\ref{represent} implies that $\rho^\natural(\K_L)=FS_L/\rho$.

On the other hand, Proposition~\ref{quot} guarantees that the restriction $q_L|\K_L:\K_L\to FK_L$ is bijective. This implies that the homomorphism $h$ is bijective and hence $\rho_\preceq=\rho$.
\end{proof}

The bijectivity of the restriction $q_L|\K_L:\K_L\to FK_L$ and Theorem~\ref{kuratwords} imply:

\begin{corollary} For any finite pointed linearly ordered set $L$ the free Kuratowski monoid $FK_L$ has cardinality $$|FK_L|=|\K_L|=K(n,p)=\sum_{i=0}^{n}\sum_{j=0}^p\textstyle{\binom{i+j}{i}\binom{i+j}{j}}$$where $n=|L_-|$ and $p=|L_+|$.
\end{corollary}

Given two non-negative natural numbers $n,p$, fix any pointed linearly ordered set $L_{n,p}$ with $|(L_{n,p})_-|=n$ and $|(L_{n,p})_+|=p$ and denote the free Kuratowski monoid $FK_{L_{n,p}}$ by $FK_{n,p}$.

\begin{proposition}\label{FK22} For any pointed linearly ordered set $L$ and any distinct elements $x,y\in FK_L$ there is a morphism of pointed linearly ordered sets $f:L\to L_{2,2}$ such that $\bar f(x)\ne \bar f(y)$. This implies that $FK_L$ embeds into some power of the free Kuratowski monoid $FK_{2,2}$.
\end{proposition}

\begin{proof} Enlarge the pointed linearly ordered set $L$ to a $*$-linearly ordered set $L^*$ and denote by $e:L\to L^*$ the identity embedding. Let $\tilde e:FS_L\to FS_{L^*}$ be the unique semigroup homomorphism extending $e$. It is clear that $\tilde e$ is an injective map.

By Proposition~\ref{quot}, the restriction $q_L|\K_L:\K_L\to FK_L$ is bijective. So, we can find Kuratowski words $u,v\in\K_L$ such that $q_L(u)=x$ and $q_L(v)=y$.
By Theorem~\ref{separate}, for the Kuratowski words $u,v\in\K_L\subset \K_{L^*}$ there exist a 2-topological space $\mathbf X$ and a $*$-morphism $g:L^*\to L(\mathbf X)$ such that $\hat g(u)\ne\hat g(v)$ where $\hat g:FS_{L^*}\to \IK(\mathbf X)$ is the unique semigroup homomorphism extending the $*$-morphism $g$. Moreover, the proof of Theorem~\ref{separate} guarantees that the linear generating set $L(\mathbf X)$ of the 2-topological space $X$ is isomorphic to the $*$-linearly ordered set $L_{2,2}$. So, there exists a (unique) bijective $*$-morphism $\iota:L_{2,2}\to L(\mathbf X)$. Let $f=\iota^{-1}\circ g:L^*\to L_{2,2}$. The commutativity of the following diagram
$$\xymatrix{
L\ar[r]^e\ar[d]&L^*\ar[d]\ar[r]_f\ar@/^20pt/[rr]^g&L_{2,2}\ar[r]_\iota&L(\mathbf X)\ar[dd]\\
\K_L\ar[r]^{\tilde e|\K_L}\ar[d]&\K_{L^*}\ar[d]\\
FS_L\ar[d]_{q_L}\ar[r]^{\tilde e}&FS_{L^*}\ar@/^20pt/[rr]^{\hat g} \ar[d]_{q_{L^*}}\ar[r]_<<<<{\hat f}&FK_{2,2}\ar[r]_{\bar\iota}&\IK(\mathbf X)\\
FK_L\ar[r]^{\bar e}&FK_{L^*}\ar[ru]_{\bar f}
}
$$
and the inequality $\hat g\circ\tilde e(u)\ne\hat g\circ\tilde e(v)$ imply that $\bar f(x)\ne\bar f(y)$.
\end{proof}

In fact, Proposition~\ref{FK22} can be improved as follows.

\begin{proposition}\label{FK21} For any pointed linearly ordered set $L$ and any distinct elements $x,y\in FK_L$ there is a pair $(n,p)\in\{(1,2),(2,1)\}$ and a morphism of pointed linearly ordered sets $f:L\to L_{n,p}$ such that $\bar f(x)\ne \bar f(y)$. This implies that $FK_L$ embeds into some power of the partially ordered monoid $FK_{1,2}\times FK_{2,1}$.
\end{proposition}

\begin{proof} Because of Proposition~\ref{FK22}, it suffices to prove that the points of the free Kuratowski monoid $FK_{2,2}$ can be separated by monotone homomorphisms into the Kuratowski monoids $FK_{1,2}$ and $FK_{1,2}$. Write the $*$-linearly ordered set $L_{2,2}$ as $L_{2,2}=\{\breve \tau_0,\breve\tau_1,1,\bar\tau_1,\bar\tau_0\}$ for some elements $$\breve \tau_0<\breve\tau_1<1<\bar\tau_1<\bar\tau_0.$$
By analogy the pointed linearly ordered sets $L_{1,2}$ and $L_{2,1}$ can be written as $L_{1,2}=\{\breve\tau,1,\bar\tau_1,\bar\tau_0\}$ and $L_{2,1}=\{\breve\tau_0,\breve\tau_1,1,\bar\tau\}$.
Consider the four surjective monotone morphisms $$h_{12}:L_{2,2}\to L_{1,2},\;h_{23}:L_{2,2}\to L_{1,2},\;h_{34}:L_{2,2}\to L_{2,1}\mbox{ and }h_{45}:L_{2,2}\to L_{2,1}$$ such that $$h_{12}(\breve\tau_0)=h_{12}(\breve\tau_1)=\breve\tau,\;h_{23}(\breve\tau_1)=h_{23}(1)=1,\; h_{34}(1)=h_{34}(\bar\tau_1)=1, \mbox{ \ and \ }h_{45}(\bar\tau_1)=h_{45}(\bar\tau_0)=\bar\tau.$$

By Proposition~\ref{quot}, the free Kuratowski monoid $FK_{2,2}$ can be identified with the 63-element set $\K_{L_{2,2}}$ of Kuratowski words in the alphabet $L_{2,2}$:
$$\boxed{\begin{aligned}
&1,\;\;\itau_0,\;\itau_1,\;\ctau_0,\;\ctau_1,\\
&\itau_0\ctau_0,\;\itau_0\ctau_1,\;\itau_1\ctau_0,\;\itau_1\ctau_1,\;
\ctau_0\itau_0,\;\ctau_0\itau_1,\;\ctau_1\itau_0,\;\ctau_1\itau_1,\;\\
&\itau_0\ctau_0\itau_0,\; \itau_0\ctau_0\itau_1,\; \itau_0\ctau_1\itau_0,\;\itau_0\ctau_1\itau_1,\;
\itau_1\ctau_0\itau_0,\; \itau_1\ctau_0\itau_1,\; \itau_1\ctau_1\itau_0,\;\itau_1\ctau_1\itau_1,\;\\
&\ctau_0\itau_0\ctau_0,\; \ctau_0\itau_0\ctau_1,\; \ctau_0\itau_1\ctau_0,\; \ctau_0\itau_1\ctau_1,\;
\ctau_1\itau_0\ctau_0,\; \ctau_1\itau_0\ctau_1,\; \ctau_1\itau_1\ctau_0,\; \ctau_1\itau_1\ctau_1,\;\\
&\itau_0\ctau_0\itau_0\ctau_1,\;\itau_0\ctau_0\itau_1\ctau_1,\;
\itau_1\ctau_0\itau_0\ctau_0,\;\itau_1\ctau_0\itau_0\ctau_1,\;\itau_1\ctau_0\itau_1\ctau_1,\;
\itau_1\ctau_1\itau_0\ctau_0,\;\itau_1\ctau_1\itau_0\ctau_1,\\
&\ctau_0\itau_0\ctau_0\itau_1,\;\ctau_0\itau_0\ctau_1\itau_1,\;
\ctau_1\itau_0\ctau_0\itau_0,\;\ctau_1\itau_0\ctau_0\itau_1,\;\ctau_1\itau_0\ctau_1\itau_1,\;
\ctau_1\itau_1\ctau_0\itau_0,\;\ctau_1\itau_1\ctau_0\itau_1,\\
&\itau_1\ctau_0\itau_0\ctau_0\itau_1,\; \itau_1\ctau_1\itau_0\ctau_1\itau_1,\;
\itau_0\ctau_0\itau_0\ctau_1\itau_1,\; \itau_1\ctau_1\itau_0\ctau_0\itau_0,\; \itau_1\ctau_0\itau_0\ctau_1\itau_1,\;\itau_1\ctau_1\itau_0\ctau_0\itau_1,\; \\
&\ctau_1\itau_0\ctau_0\itau_0\ctau_1,\; \ctau_1\itau_1\ctau_0\itau_1\ctau_1,\;
\ctau_0\itau_0\ctau_0\itau_1\ctau_1,\; \ctau_1\itau_1\ctau_0\itau_0\ctau_0,\; \ctau_1\itau_0\ctau_0\itau_1\ctau_1,\;\ctau_1\itau_1\ctau_0\itau_0\ctau_1,\;\\
&\itau_1\ctau_1\itau_0\ctau_0\itau_1\ctau_1,\; \itau_1\ctau_0\itau_0\ctau_0\itau_1\ctau_1,\;\itau_1\ctau_1\itau_0\ctau_0\itau_0\ctau_1,\\
&\ctau_1\itau_1\ctau_0\itau_0\ctau_1\itau_1,\; \ctau_1\itau_0\ctau_0\itau_0\ctau_1\itau_1,\;\ctau_1\itau_1\ctau_0\itau_0\ctau_0\itau_1,\;\\
&\itau_1\ctau_1\itau_0\ctau_0\itau_0\ctau_1\itau_1,\;\; \ctau_1\itau_1\ctau_0\itau_0\ctau_0\itau_1\ctau_1.
\end{aligned}
}$$

In the following two lists we write the pairs $(h_{12}(w),h_{45}(w))$ and $(h_{23}(w),h_{34}(w))$ for the Kuratowski words $w$ from the above list. Analyzing these two lists we can see that the quadruples
$(h_{12}(w),h_{45}(w),h_{23}(w),h_{34}(w))$, $w\in\K_{L_{2,2}}$, are pairwise distinct, which means that the elements of the free Kuratowski monoid $FK_{2,2}$ are separated by the homomorphism $(h_{12},h_{23},h_{34},h_{45}):L_{2,2}\to L_{1,2}^2\times L_{2,1}^2$.

$$\boxed{\small
\begin{aligned}
&\hskip150pt\mbox{The homomorphism $(h_{12},h_{45})$:}\\
&(1,1),\;\;(\itau,\itau_0),\;(\itau,\itau_1),\;(\ctau_0,\ctau)\;(\ctau_1,\ctau),\\
&(\itau\ctau_0,\itau_0\ctau)\;(\itau\ctau_1,\itau_0\ctau),\;(\itau\ctau_0,\itau_1\ctau),\;(\itau\ctau_1,\itau_1\ctau),\;
(\ctau_0\itau,\ctau\itau_0),\;(\ctau_0\itau,\ctau\itau_1),\;(\ctau_1\itau,\ctau\itau_0),\;(\ctau_1\itau,\ctau\itau_1),\;\\
&(\itau\ctau_0\itau,\itau_0\ctau\itau_0),\; (\itau\ctau_0\itau,\itau_0\ctau\itau_1),\; (\itau\ctau_1\itau,\itau_0\ctau\itau_0),\;(\itau\ctau_1\itau,\itau_0\ctau\itau_1),\;
(\itau\ctau_0\itau,\itau_1\ctau\itau_0),\; (\itau\ctau_0\itau,\itau_1\ctau\itau_1),\; (\itau\ctau_1\itau,\itau_1\ctau\itau_0),\;(\itau\ctau_1\itau,\itau_1\ctau\itau_1),\\
&(\ctau_0\itau\ctau_0,\ctau\itau_0\ctau),\; (\ctau_0\itau\ctau_1,\ctau\itau_0\ctau),\; (\ctau_0\itau\ctau_0,\ctau\itau_1\ctau),\; (\ctau_0\itau\ctau_1,\ctau\itau_1\ctau),\;
(\ctau_1\itau\ctau_0,\ctau\itau_0\ctau),\; (\ctau_1\itau\ctau_1,\ctau\itau_0\ctau),\; (\ctau_1\itau\ctau_0,\ctau\itau_1\ctau),\; (\ctau_1\itau\ctau_1,\ctau\itau_1\ctau),\;\\
&(\itau\ctau_0\itau\ctau_1,\itau_0\ctau),\;(\itau\ctau_0\itau\ctau_1,\itau_0\ctau),\;
(\itau\ctau_0,\itau_1\ctau\itau_0\ctau),\;(\itau\ctau_0\itau\ctau_1,\itau_1\ctau\itau_0\ctau),\;
(\itau\ctau_0\itau\ctau_1,\itau_1\ctau),\;
(\itau\ctau_0,\itau_1\ctau\itau_0\ctau),\;(\itau\ctau_1,\itau_1\ctau\itau_0\ctau),\\
&(\ctau_0\itau,\ctau\itau_0\ctau\itau_1),\;(\ctau_0\itau,\ctau\itau_0\ctau\itau_1),\;
(\ctau_1\itau\ctau_0\itau,\ctau\itau_0),\;(\ctau_1\itau\ctau_0\itau,\ctau\itau_0\ctau\itau_1),\;
(\ctau_1\itau,\ctau\itau_0\ctau\itau_1),\;
(\ctau_1\itau\ctau_0\itau,\ctau\itau_0),\;(\ctau_1\itau\ctau_0\itau,\ctau\itau_1),\\
&(\itau\ctau_0\itau,\itau_1\ctau\itau_0\ctau\itau_1),\; (\itau\ctau_1\itau,\itau_1\ctau\itau_0\ctau\itau_1),\;
(\itau\ctau_0\itau,\itau_0\ctau\itau_1),\; (\itau\ctau_0\itau,\itau_1\ctau\itau_0),\; (\itau\ctau_0\itau,\itau_1\ctau\itau_0\ctau\itau_1),\;
(\itau\ctau_0\itau,\itau_1\ctau\itau_0\ctau\itau_1),\; \\
&(\ctau_1\itau\ctau_0\itau\ctau_1,\ctau\itau_0\ctau),\; (\ctau_1\itau\ctau_0\itau\ctau_1,\ctau\itau_1\ctau),\;
(\ctau_0\itau\ctau_1,\ctau\itau_0\ctau),\; (\ctau_1\itau\ctau_0,\ctau\itau_0\ctau),\; (\ctau_1\itau\ctau_0\itau\ctau_1,\ctau\itau_0\ctau),\;
(\ctau_1\itau\ctau_0\itau\ctau_1,\ctau\itau_0\ctau),\;\\
&(\itau\ctau_0\itau\ctau_1,\itau_1\ctau\itau_0\ctau),\; (\itau\ctau_0\itau\ctau_1,\itau_1\ctau\itau_0\ctau),\;
(\itau\ctau_0\itau\ctau_1,\itau_1\ctau\itau_0\ctau),\\
&(\ctau_1\itau\ctau_0\itau,\ctau\itau_0\ctau\itau_1),\; (\ctau_1\itau\ctau_0\itau,\ctau\itau_0\ctau\itau_1),\;
(\ctau_1\itau\ctau_0\itau,\ctau\itau_0\ctau\itau_1),\;\\
&(\itau\ctau_0\itau,\itau_1\ctau\itau_0\ctau\itau_1),\;\; (\ctau_1\itau\ctau_0\itau\ctau_1,\ctau\itau_0\ctau).
\end{aligned}
}$$

$$\boxed{\small
\begin{aligned}
&\hskip150pt\mbox{The homomorphism $(h_{23},h_{34})$:}\\
&(1,1)\;\;(\itau_0,\itau_0),\;(1,\itau_1),\;(\ctau_0,\ctau_0),\;(\ctau_1,1),\\
&(\itau_0\ctau_0,\itau_0\ctau_0),\;(\itau_0\ctau_1,\itau_0),\;(\ctau_0,\itau_1\ctau_0),\;
(\ctau_1,\itau_1),\;
(\ctau_0\itau_0,\ctau_0\itau_0),\;(\ctau_0,\ctau_0\itau_1),\;(\ctau_1\itau_0,\itau_0),\;
(\ctau_1,\itau_1),\;\\
&(\itau_0\ctau_0\itau_0,\itau_0\ctau_0\itau_0),\; (\itau_0\ctau_0,\itau_0\ctau_0\itau_1),\; (\itau_0\ctau_1\itau_0,\itau_0),\;(\itau_0\ctau_1,\itau_0),\;
(\ctau_0\itau_0,\itau_1\ctau_0\itau_0),\; (\ctau_0,\itau_1\ctau_0\itau_1),\; (\ctau_1\itau_0,\itau_0),\;(\ctau_1,\itau_1),\;\\
&(\ctau_0\itau_0\ctau_0,\ctau_0\itau_0\ctau_0),\; (\ctau_0\itau_0\ctau_1,\ctau_0\itau_0),\; (\ctau_0,\ctau_0\itau_1\ctau_0),\; (\ctau_0,\ctau_0\itau_1),\;
(\ctau_1\itau_0\ctau_0,\itau_0\ctau_0),\; (\ctau_1\itau_0\ctau_1,\itau_0),\; (\ctau_1\ctau_0,\itau_1\ctau_0),\; (\ctau_1,\itau_1),\;\\
&(\itau_0\ctau_0\itau_0\ctau_1,\itau_0\ctau_0\itau_0),\;
(\itau_0\ctau_0,\itau_0\ctau_0\itau_1),\;
(\ctau_0\itau_0\ctau_0,\itau_1\ctau_0\itau_0\ctau_0),\;
(\ctau_0\itau_0\ctau_1,\itau_1\ctau_0\itau_0),\;
(\ctau_0,\itau_1\ctau_0\itau_1),\;
(\ctau_1\itau_0\ctau_0,\itau_0\ctau_0),\; (\ctau_1\itau_0\ctau_1,\itau_0),\\
&(\ctau_0\itau_0\ctau_0,\ctau_0\itau_0\ctau_0\itau_1),\;
(\ctau_0\itau_0\ctau_1,\ctau_0\itau_0),\;
(\ctau_1\itau_0\ctau_0\itau_0,\itau_0\ctau_0\itau_0),\;
(\ctau_1\itau_0\ctau_0,\itau_0\ctau_0\itau_1),\;
(\ctau_1\itau_0\ctau_1,\itau_0),\;
(\ctau_0\itau_0,\itau_1\ctau_0\itau_0),\;
(\ctau_0,\itau_1\ctau_0\itau_1),\\
&(\ctau_0\itau_0\ctau_0,\itau_1\ctau_0\itau_0\ctau_0\itau_1),\;
(\ctau_1\itau_0\ctau_1,\itau_0),\;
(\itau_0\ctau_0\itau_0\ctau_1,\itau_0\ctau_0\itau_0),\; (\ctau_1\itau_0\ctau_0\itau_0,\itau_0\ctau_0\itau_0),\; (\ctau_0\itau_0\ctau_1,\itau_1\ctau_0\itau_0),\;
(\ctau_1\itau_0\ctau_0,\itau_0\ctau_0\itau_1),\; \\
&(\ctau_1\itau_0\ctau_0\itau_0\ctau_1,\itau_0\ctau_0\itau_0),\;
(\ctau_0,\itau_1\ctau_0\itau_1),\;
(\ctau_0\itau_0\ctau_0,\ctau_0\itau_0\ctau_0\itau_1),\; (\ctau_0\itau_0\ctau_0,\itau_1\ctau_0\itau_0\ctau_0),\; (\ctau_1\itau_0\ctau_0,\itau_0\ctau_0\itau_1),\;
(\ctau_0\itau_0\ctau_1,\itau_1\ctau_0\itau_0),\;\\
&(\ctau_1\itau_0\ctau_0,\itau_0\ctau_0\itau_1),\; (\ctau_0\itau_0\ctau_0,\itau_1\ctau_0\itau_0\ctau_0\itau_1),\;
(\ctau_1\itau_0\ctau_0\itau_0\ctau_1,\itau_0\ctau_0\itau_0),\\
&(\ctau_0\itau_0\ctau_1,\itau_1\ctau_0\itau_0),\; (\ctau_1\itau_0\ctau_0\itau_0\ctau_1,\itau_0\ctau_0\itau_0),\;
(\ctau_0\itau_0\ctau_0,\itau_1\ctau_0\itau_0\ctau_0\itau_1),\;\\
&(\ctau_1\itau_0\ctau_0\itau_0\ctau_1,\itau_0\ctau_0\itau_0),\;\; (\ctau_0\itau_0\ctau_0,\itau_1\ctau_0\itau_0\ctau_0\itau_1).
\end{aligned}
}$$
\end{proof}

Proposition~\ref{FK21} shows that the homomorphisms into free Kuratowski monoids $FK_{n,p}$ for $n+p\le 3$ completely determine the structure of any free Kuratowski monoid. The following diagram shows the order structure of the free Kuratowski monoids $FK_{1,1}$ with the linear generating set $L_{1,1}=\{a,1,x\}$. We identify the elements of $FK_{1,1}$ with the Kuratowski words in the alphabet $L_{1,1}$. For two Kuratowski words $u,v$ an arrow $u\to v$ indicated that $u\le v$ in $FK_{1,1}$.
$$
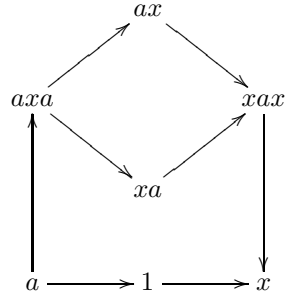
\xymatrix{
&ax\ar[rd]&\\
axa\ar[ru]\ar[rd]&&xax\ar[dd]\\
&xa\ar[ru]\\
a\ar[uu]\ar[r]&1\ar[r]&x
}$$

The following diagram shows the order structure of the free Kuratowski monoids $FK_{2,1}$ with the linear generating set $L_{2,1}=\{a,b,1,x\}$:
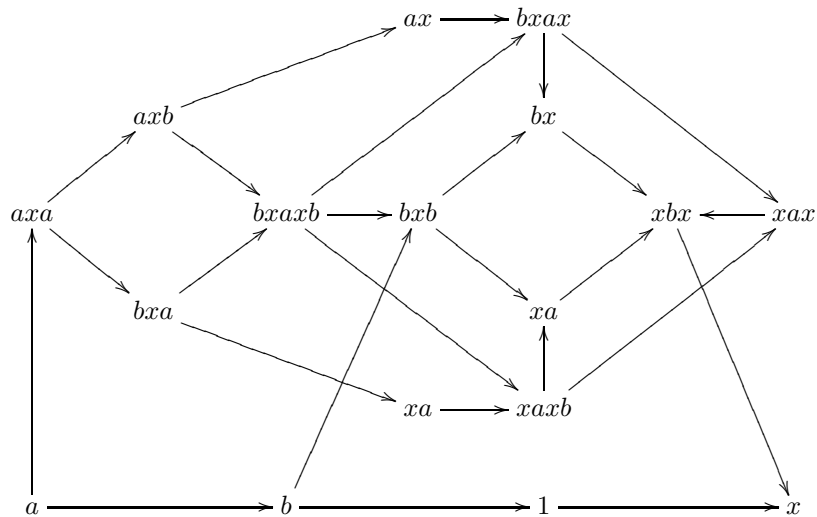
$$\xymatrix{
&&&ax\ar[r]&bxax\ar[d]\ar[ddrr]\\
&axb\ar[rd]\ar[rru]&&&bx\ar[rd]\\
axa\ar[ru]\ar[rd]&&bxaxb\ar[r]\ar[rruu]\ar[rrdd]&bxb\ar[ru]\ar[rd]&&xbx\ar[dddr]&xax\ar[l]\\
&bxa\ar[ru]\ar[rrd]&&&xa\ar[ru]\\
&&&xa\ar[r]&xaxb\ar[u]\ar[uurr]\\
a\ar[uuu]\ar[rr]&&b\ar[ruuu]\ar[rr]&&1\ar[rr]&&x
}$$
\smallskip

The free Kuratowski monoid $K_{1,2}$ is isomorphic to the free Kuratowski monoid $K_{2,1}$ endowed with the reversed partial order.

Proposition~\ref{FK21} has an interesting application.

\begin{proposition} All elements of any Kuratowski monoid $K$ are idempotents.
\end{proposition}

\begin{proof} Let $L$ be the linear generating set of the Kuratowski monoid $K$. Since $K$ is a quotient monoid of the free Kuratowski monoid $FK_L$, it suffices to check that all elements of $FK_L$ are idempotents. By Proposition~\ref{FK21}, the free Kuratowski monoid $FK_L$ embeds into some power of the monoid $FK_{1,2}\times FK_{2,1}$. So, it suffices to check that each element of the free Kuratowski monoids $FK_{1,2}$ and $FK_{2,1}$ is an idempotent. This can be seen by a direct verification of each of 17 elements of $FK_{1,2}$ (or its algebraically isomorphic copy $FK_{2,1}$).
\end{proof}

\end{document}